\documentclass[final]{amsart}

\usepackage[utf8]{inputenc} 

\usepackage{graphicx} 

\usepackage[foot]{amsaddr} 
\usepackage{verbatim} 
\usepackage{amssymb}
\usepackage{amsthm}
\usepackage{amsmath}
\usepackage{amsfonts}
\usepackage{latexsym}
\usepackage{mathtools} 
\usepackage{graphics}
\usepackage{float} 
\usepackage{enumitem} 
\usepackage[english]{babel} 

\usepackage{esint}

\usepackage{calrsfs}

    \usepackage[backref=none]{hyperref} 
     \usepackage[usenames,dvipsnames]{color}
\definecolor{pennblue}{RGB}{1,31,91} 
\definecolor{pennred}{RGB}{153,0,0}   
     \definecolor{MyDarkBlue}{rgb}{0,0.1,0.7}
     \hypersetup{pdfborder={0 0 0},colorlinks,breaklinks=true,
       urlcolor={pennblue},citecolor={pennblue},linkcolor={pennblue}
     }

\usepackage[color]{showkeys}

\usepackage{color}

\usepackage{dsfont}

\theoremstyle{plain}
\newtheorem{theorem}{Theorem}

\newtheorem{definition}[theorem]{Definition}

\newtheorem{remark}[theorem]{Remark}

\numberwithin{equation}{section}
\numberwithin{theorem}{section}

\allowdisplaybreaks

\newcommand{\eqdef }{\overset{\mbox{\tiny{def}}}{=}}

\newcommand{\rone}{\mathbb{R}}

\title[opPINN: Physics-Informed Neural Network with operator learning]{opPINN: Physics-Informed Neural Network with operator learning to approximate solutions to the Fokker-Planck-Landau equation}

\author[J. Y. Lee]{Jae Yong Lee$^1$}
\address{$^1$Center for Artificial Intelligence and Natural Sciences, Korea Institute for Advanced Study (KIAS), Seoul 02455, Republic of Korea. \href{mailto:jaeyong@kias.re.kr}{jaeyong@kias.re.kr} }

\author[J. Jang]{Juhi Jang$^2$}
\address{$^2$Department of Mathematics, University of Southern California, Los Angeles, CA 90089, USA. \href{mailto:juhijang@usc.edu}{juhijang@usc.edu}}

\author[H. J. Hwang]{Hyung Ju Hwang$^{3,\dagger}$}
\address{$^3$Department of Mathematics, Pohang University of Science and Technology (POSTECH), Pohang 37673, Republic of Korea. \href{mailto:hjhwang@postech.ac.kr}{hjhwang@postech.ac.kr}}
\address{$^\dagger$Corresponding author.}

\AtBeginDocument{%
   \def\MR#1{}
}

\begin{document}



\let\thefootnote\relax\footnotetext{
    \textit{Key words and phrases.} Fokker-Planck-Landau equation, Kinetic theory of gases, Physics-informed neural network (PINN), Operator learning, and Deep learning.}
\addtocounter{footnote}{-1}\let\thefootnote\svthefootnote

\begin{abstract}We propose a hybrid framework opPINN: physics-informed neural network (PINN) with operator learning for approximating the solution to the Fokker-Planck-Landau (FPL) equation. The opPINN framework is divided into two steps: Step 1 and Step 2. After the operator surrogate models are trained during Step 1, PINN can effectively approximate the solution to the FPL equation during Step 2 by using the pre-trained surrogate models. The operator surrogate models greatly reduce the computational cost and boost PINN by approximating the complex Landau collision integral in the FPL equation. The operator surrogate models can also be combined with the traditional numerical schemes. It provides a high efficiency in computational time when the number of velocity modes becomes larger. Using the opPINN framework, we provide the neural network solutions for the FPL equation under the various types of initial conditions, and interaction models in two and three dimensions. Furthermore, based on the theoretical properties of the FPL equation, we show that the approximated neural network solution converges to the a priori classical solution of the FPL equation as the pre-defined loss function is reduced.
\end{abstract}

\setcounter{tocdepth}{2}

\maketitle
\tableofcontents

\thispagestyle{empty}

\section{Introduction}
\subsection{Motivation and main results}
The Fokker-Planck-Landau (FPL) equation  
is one of fundamental kinetic models commonly used to describe collisions among charged particles with long-range interactions. Due to its physical importance, many numerical approaches have been developed. Nonlinearity and the high dimensionality of variables make the FPL equation difficult to simulate efficiently and accurately. One of the main challenges to simulate the FPL equation comes from a collision integral operator. The complex three-fold integro-differential operator in the velocity domain must be handled carefully as it is related to macroscopic physical quantities, such as the conservation of total mass, momentum, and kinetic energy.

Various numerical methods have been studied for the simulation of the FPL equation \cite{MR3202241}. Among others, a fast spectral method using the fast Fourier transform has been proposed to obtain a higher accuracy and efficiency \cite{MR1906573,MR1795398}. It reduces the complexity of evaluating the quadratic Landau collision operator $Q(f,f)$ from $O(N^2)$ to $O(N\log N)$, where $N$ is the number of modes in the discrete velocity space. However, it still takes a huge computation cost as $N$ increases.

Recently, deep learning based methods have been developed to solve PDEs with lots of advantages. Especially, a physics-informed neural network (PINN) \cite{MR3881695} learns the neural network parameters to minimize the residual of PDEs in the least square sense. PINN can approximate a mesh-free solution by learning with grid sampling from a domain of PDEs without any discretization of the domain. Furthermore, PINN can calculate the differentiation by using a powerful technique Automatic Differentiation that enables us to take the derivative of any order of functions easily.

However, PINN has some difficulties in applying directly to complex PDEs. When the PDEs contain complex integrals, PINN needs to sample grid points for integration as well as grid points in a given domain. It requires lots of computational costs to use PINN for approximating solutions to the equations. Furthermore, the cost can be larger when the dimension of the equation becomes large. It is hard to use PINN directly to deal with the kinetic equation with a complex collision operator.

To overcome the difficulties described above, we propose a hybrid framework opPINN, which efficiently combines PINN with an operator learning approach to approximate a mesh-free solution to the FPL equation. Operator learning is a deep learning approach to learn a mapping from the infinite-dimensional functions to infinite-dimensional functions. Many neural network structures containing convolutional neural networks are developed as a surrogate model to learn operators in various problems for PDEs \cite{li2020fourier,lu2019deeponet,MR3800689}. The proposed opPINN framework uses the operator learning to approximate the complex operator in the Landau collision integral of the FPL equation.

The opPINN framework is divided into two steps: Step 1 and Step 2. The operator surrogate models are pre-trained during Step 1. Using the pre-trained operator surrogate models, the solution to the FPL equation is approximated effectively via PINN during Step 2 with reduced computational costs. Furthermore, the surrogate model takes a great advantage of flexibility in that it can also be combined with already built existing traditional numerical schemes, such as finite different methods (FDMs) and temporal discretization schemes, to simulate the FPL equation. It can be very useful in computational time as the number of velocity modes $N$ becomes larger since only the inference time is required after the operator surrogate model is pre-trained.

There are many works to combine the operator learning with PINN \cite{li2021physics,wang2021learning,MR3957452}. They focus on learning a solution operator with the aid of the PDE constraints as the PINN perspective. On the contrary, the proposed opPINN focuses on approximating a solution to given integro-differential PDE using PINN with the pre-trained integration operator. Therefore, the proposed opPINN framework can take great advantages of PINN. The proposed framework generates mesh-free solutions that can be differentiated continuously with respect to each input variable. Furthermore, the approximated neural network solution preserves the positivity at all continuous time by using a softplus activation function in the last layer of our neural network structure. It is important but in general difficult to keep a positive solution at any time in numerical methods for the kinetic equations. 

Many studies \cite{holloway2021acceleration,miller2021encoder,xiao2021using} use deep learning to approximate the integral collision operator in the FPL equation or Boltzmann equation. Authors in 
\cite{miller2021encoder,xiao2021using} use operator learning directly to approximate the collision operator. However, they have some weaknesses in the experiment results, such as the time evolution of the approximated solution for various initial conditions, and the trend of the theoretical equilibrium of the approximated solution. Authors in \cite{holloway2021acceleration} approximate the solution to the Boltzmann equation using Bhatnagar-Gross-Krook (BGK) solver and the prior knowledge of the Maxwellian equilibrium instead of directly solving the Boltzmann equation. They also only focus on the reduced distribution solution which has only one or two dimensions in the velocity variable.

Furthermore, those studies also do not provide the theoretical guarantee that the solution of the FPL equation and the Boltzmann equation can be well approximated using a neural network structure. The theoretical evidence of the proposed equation solver is important in numerical analysis. In this work, we provide a theoretical support that the opPINN framework can approximate the solution to the FPL equation by using the theoretical backgrounds on the FPL equation. 

The main contributions of the study are as follows.

\begin{itemize}[topsep=5pt]
    \item We propose a hybrid framework opPINN using the operator learning approach with PINN (Figure \ref{fig:framework}). After the operator surrogate models are trained during Step 1, the solution to the FPL equation is effectively approximated using PINN during Step 2 with the pre-trained operator surrogate models. The operator surrogate models greatly reduce the computational cost and boost PINN by approximating the Landau collision integral in the FPL equation effectively. To the best of authors’ knowledge, this is the first attempt to use the operator learning to make PINN algorithm more effective for complex PDEs.

    \item The operator surrogate model takes a great advantage of flexibility in that it can also be combined with existing traditional numerical schemes. Combining the operator surrogate models with FDM and temporal discretization schemes, the solution to the FPL equation can be approximated (Figure \ref{fig:operator_euler}). It has a high efficiency in computational time for approximating the Landau collision term when the number of velocity modes $N$ becomes larger (Figure \ref{fig:operator_speed}). This is because the surrogate model only needs to infer for approximating the collision integral operator after the model is pre-trained.

    \item The opPINN framework generates the mesh-free neural network solutions for the FPL equation, which are continuous-in-time smooth and preserve the positivity at any time. Using the opPINN framework, we provide the neural network solutions for the homogeneous FPL equation under the various types of initial conditions, and interaction models in two and three dimensions. We show that the neural network solutions converge pointwisely to the equilibrium based on the kinetic theory. Furthermore, we provide the graphs of macroscopic physical quantities including the mass, the momentum, the kinetic energy, and the entropy.

    \item We also provide the theoretical support that the opPINN framework can approximate the solution to the FPL equation. Based on the theoretical properties of the FPL equation, the approximated neural network solution to the FPL equation using the proposed opPINN framework converges to the a priori classical solution to the FPL equation as the pre-defined total loss is reduced.
\end{itemize}

\subsection{The Fokker-Planck-Landau equation}The FPL equation is a kinetic model describing the particles in terms of a distribution function $f(t,\textbf{x},\textbf{v})$ of particles at time $t$, position $\textbf{x}$ with velocity $\textbf{v}$. The $d-$dimensional ($d=$2,3) FPL equation is
\begin{equation}\label{FPL_nonhomo}
    \partial_tf(t,\textbf{x},\textbf{v}) + \textbf{v}\cdot\nabla_xf(t,\textbf{x},\textbf{v})=Q(f,f)\eqdef\nabla_\textbf{v}\cdot\left(\int_{\rone^d}d\textbf{v}_*\Phi(\textbf{v}-\textbf{v}_*)\left(f_*\nabla f-f\nabla f_*\right)\right),
\end{equation}
where $\textbf{x},\textbf{v}\in\mathbb{R}^d$ and $f_*\eqdef f(t,\textbf{x},\textbf{v}_*)$ and $\Phi(\textbf{z})$ is the collision kernel that is given by the following $d\times d$ matrix 
\begin{equation}\label{kernel}
    \Phi(\textbf{z})=\Lambda|\textbf{z}|^\gamma(|\textbf{z}|^2I_d-\textbf{z}\otimes\textbf{z}).
\end{equation}
Here, $\gamma$ is the index of the power of the distance. $\gamma>0$ corresponds to the hard potential model and $\gamma<0$ corresponds to the soft potential model. The model of Maxwell molecules is the case $\gamma=0$ and the model with Coulomb interactions is the case $\gamma=-3$ when $d=3$.

In this work, we study the $d-$dimensional homogeneous FPL equation given as
\begin{equation}\label{FPL}
    \partial_tf(t,\textbf{v})=Q(f,f).
\end{equation}
The Landau collision term $Q(f,f)$ can be rewritten as the following form
\begin{equation}\label{landau}
    Q(f,f)=\nabla_\textbf{v}\cdot\left(D(f)\nabla f-F(f)f\right),
\end{equation}
where
\begin{equation}\label{Df}
    D(f)=\int_{\rone^d}d\textbf{v}_*\;\Phi(\textbf{v}-\textbf{v}_*)f(\textbf{v}_*)\in\mathbb{R}^{d\times d},
\end{equation}
and
\begin{equation}\label{Ff}
    F(f)=\int_{\rone^d}d\textbf{v}_*\;\Phi(\textbf{v}-\textbf{v}_*)\nabla f(\textbf{v}_*)\in\mathbb{R}^d.
\end{equation}

\subsection{Fundamental properties of physical quantities and the equilibrium state}\label{properties}The collision operator of the FPL equation has a similar algebraic structure to that of the Boltzmann operator. It leads to physical properties of the FPL collision operator. We denote the density $\rho$, mean velocity $u$, and temperature $T$ as
\begin{equation}
    \rho\eqdef\int_{\rone^d}d\textbf{v}f,\quad \textbf{u}\eqdef\frac{1}{\rho}\int_{\rone^d}d\textbf{v}\;\textbf{v}f,\quad T\eqdef\frac{1}{\rho d}\int_{\rone^d}d\textbf{v}\;|\textbf{u}-\textbf{v}|^2f.
\end{equation} It is well known \cite{MR1392006} that the FPL equation \eqref{FPL} has the conservation of mass, momentum and kinetic energy given by
\begin{equation}\label{physical_quantities}
    \text{Mass}(t)\eqdef\rho,\quad \text{Mom}(t)\eqdef\rho \textbf{u},\quad \text{KE}(t)\eqdef\rho|\textbf{u}|^2+\rho dT.
\end{equation}
Also, the entropy of the system
\begin{equation}\label{entropy}
    \text{Ent}(t)\eqdef\int_{\rone^d}d\textbf{v}\;f\log f
\end{equation} is a non-increasing function, i.e.
\begin{equation}\label{entropy_nonincrease}
    \frac{d}{dt}\text{Ent}(t)\leq0.
\end{equation}Moreover, the equilibrium state of the FPL equation is known as the Maxwellian
\begin{equation}\label{maxwellian}
    M_{\rho,u,T,d}(\textbf{v})\eqdef\frac{\rho}{(2\pi T)^{d/2}}\exp\left(-\frac{|\textbf{v}-\textbf{u}|^2}{2T}\right).
\end{equation}

\subsection{Related works}The FPL operator is originally obtained from the limit of the Boltzmann operator when all binary collisions are grazing \cite{MR1055522,MR1167768,MR1165528}. This limit was first derived by Landau \cite{landau1958kinetic}. There are many theoretical studies for the FPL equation. Among others, Desvillettes and Villani in \cite{MR1737547,MR1737548} showed the existence, uniqueness and smoothness of weak solutions to the spatially homogeneous FPL equation when $\gamma>0$. Chen in \cite{MR2745513} considered the regularity of weak solutions to the Landau equation. The exponential convergence to the equilibrium to the FPL equation was proved in \cite{MR3407515} for the hard potential case. For the soft potential case ($\gamma<0$), Fournier and Guerin \cite{MR2502525} proved the uniqueness of weak solution to the Landau equation using the probabilistic techniques. In the framework of classical solutions, Guo in \cite{MR1946444} established the global-in-time existence and stability in a periodic box when the initial data is close to Maxwellian. For more recent analytical studies on the FPL equation, we refer to the \cite{MR3375485,MR3778645,MR3365830,MR3369941,MR2718931,MR3599518,MR4076068,MR3158719} and references therein.

A lot of numerical methods have been developed for the FPL equation to evaluate the quadratic collision operator $Q(f,f)$ that has the complexity $O(N^2)$, where $N$ is the number of modes in the discrete velocity space. In the past decades, many studies proposed the probabilistic method known as the particle method \cite{MR3645392} that uses the linear combination of Dirac delta functions for the particle locations. Carrillo et al. in \cite{MR4121055,carrillo2021random} propose a novel particle method that preserve important properties of the Landau operator, such as conservation of mass, momentum, energy and the decay of entropy. There are also many deterministic methods to approximate the complex Landau collision operator \cite{MR1640174,MR1677589,MR3635847}. Lately, a numerical method based on the FDM is developed for the spatially homogeneous FPL equation in \cite{MR3654740}. The proposed schemes have been designed to satisfy the conservation of physical quantities, positivity of the solution, and the convergence to the Maxwellian equilibrium. However, many works focus on the isotropic case or the simplified problems due to the computational complexity of the three-fold collision integral. On the other hand, different schemes based on the spectral method have been studied for the FPL equation. Pareschi et al in \cite{MR1795398} first proposed the Fourier spectral method to simulate the homogeneous FPL equation using the fast Fourier transform so that the computational complexity $O(N^2)$ is reduced to $O(N\log N)$. Later, Filbet and Pareschi in \cite{MR1906573} extended the spectral method to the inhomogeneous FPL equation with one dimension in space and two dimension in velocity. The spectral method combined with a Hermite expansion has also been proposed \cite{MR4224112,MR4108209}. Filbet in \cite{filbet2020spectral} used the spectral collocation method to reduce the number of discrete convolutions. There are also other fast algorithms that reduce costs to $O(N\log N)$, such as a multigrid algorithm \cite{MR1447091} and a multipole expansion\cite{MR1606249}. There are also an asymptotic preserving scheme for the nonhomogeneous FPL equations proposed in \cite{MR2818606}.

As a computing power increases, deep learning methods have emerged for scientific computations, including PDE solvers. There are two mainstreams to solve PDEs with the aid of a deep learning, namely using a neural network directly to parametrize a solution to the PDE and using a neural network as a surrogate model to learn a solution operator. The first approach was proposed by many works \cite{lagaris1998artificial,lagaris2000neural,MR1078748,MR2282808,MR3211506}. In recent times, PINN is developed by Raissi et al. \cite{MR3881695}. They parametrize a PDE solution by minimizing a loss function that is a least-square error of the residual of PDE. The authors in \cite{MR3847747,MR3874585} show the capabilities of PINN for solving high-dimensional PDEs. Authors in \cite{MR4116803,MR4313375} study Vlasov-Poisson-Fokker-Planck equation and its diffusion limit using PINN. They also give the theoretical proof that the neural network solutions converge to the a priori classical solution to the kinetic equation as the proposed loss function is reduced. The review paper \cite{karniadakis2021physics} and its references contain the recent trends and diverse applications of the PINN. 

For the second approach, called an operator learning, a neural network is utilized as a mapping from the parameters of a given PDE to its solution or the target function.  Li et al. in \cite{li2020fourier,li2020neural} propose an iterative scheme to learn a solution operator of PDEs, called a neural operator. Lu et al. in \cite{lu2019deeponet} propose a DeepONet based on the universal approximation theorem of function operators. Authors in \cite{li2021physics,wang2021learning,MR3957452} extend the previous operator learning models by using a PDE information as an additional loss function to learn more accurate solution operators. Authors in \cite{holloway2021acceleration} approximate the solution to the Boltzmann equation using fully connected neural network with the aid of a BGK solver. Authors in \cite{miller2021encoder,xiao2021using} use the operator learning approach to simulate the FPL equation and the Boltzmann equation.

\subsection{Outline of the paper}In Section \ref{sec:method}, we will introduce the overview of the opPINN framework to approximate the FPL equation. The descriptions of the opPINN framework are divided into two steps: Step 1 and Step 2. It includes a detailed explanation of the train data and the loss functions for each step. In Section \ref{sec:convergence}, we provide the theoretical proof on the convergence of the neural network solutions to the classical solutions of the FPL equation. In Section \ref{sec:simulation}, we show our simulation results on the FPL equation under various conditions. Finally, in Section \ref{sec:conclusion}, we summarize the proposed method and the results of this paper.

\section{Methodology: Physics-Informed Neural Network with operator learning (opPINN)}\label{sec:method} The proposed opPINN framework is divided into two steps: Step 1 and Step 2 as visualized in Figure \ref{fig:framework}. During Step 1, the operators $D(\cdot):f\mapsto D(f)$ and $F(\cdot):f\mapsto F(f)$ are approximated. During Step 2, the solution $f$ of the FPL equation \eqref{FPL} is approximated using PINN. Here, the pre-trained operator surrogated models $D^{nn}(f;\theta_D)$ and $F^{nn}(f;\theta_F)$ from Step 1 are used to approximate the collision term $Q(f,f)$ in the FPL equation. We use the PyTorch library for opPINN framework. Each step is explained in the following subsections.

\begin{figure}[H]
  \includegraphics[width=\textwidth, draft=false]{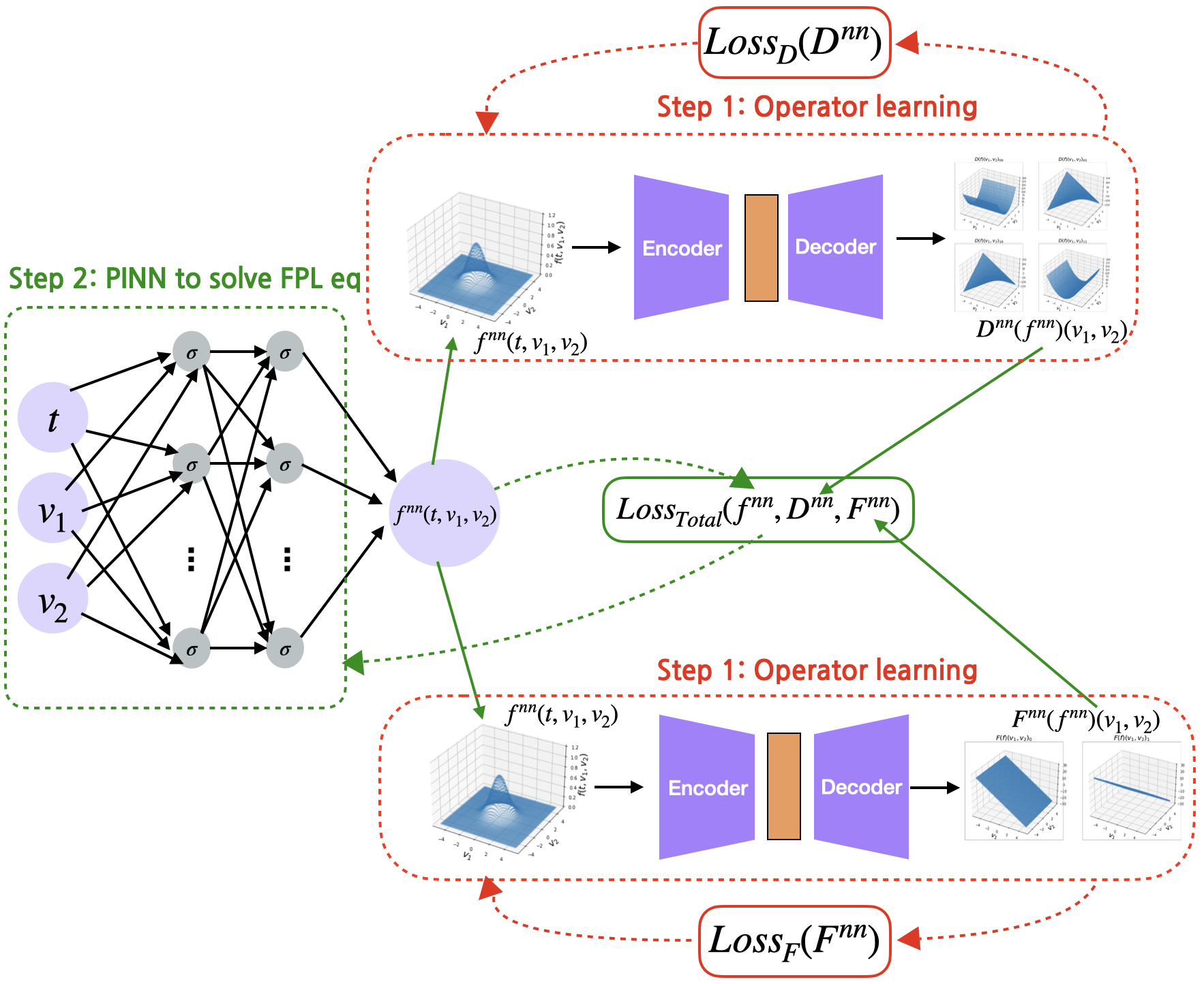}
  \caption{Overview of the proposed opPINN framework to approximate the solution to the FPL equation.}
  \label{fig:framework}
\end{figure}

\subsection{Step 1: Operator learning to approximate the collision operator}\label{subsec:step1}Many structures are used to approximate an operator using neural networks in recent studies. In particular, convolutional encoder-decoder structure is one of well-known neural network structures, which is broadly used in image-to-image regression tasks. In this study, we use two convolution encoder-decoder networks to learn operators $D(f)$ and $F(f)$. Each convolution network takes the distribution $f(t,\textbf{v})$ as input to obtain the output $D(f)$ and $F(f)$. Other structures, such as DeepONet \cite{lu2019deeponet} or Neural Operator \cite{li2020neural}, could also be available for operator learning in the future. In two-dimensional FPL equation, the solution $f(t,v_1,v_2)$ is considered as a two-dimensional image with a fixed time $t$ so that Conv 2d layers are used in the proposed model. Otherwise, Conv 3d layers are used for three-dimensional FPL equation. In the experiments, we use the $N^d$ size of image-like data $f(t,\textbf{v})$ for different $N$ to compare the spectral method with the $N$ modes of each velocity space (total $N^d$ modes). The encoder consists of 4 convolutional layers and a linear layer. The decoder consists of 4 or 5 convolution layers and 3 upsampling layers. The number of the convolution layers in the decoder depends on the dimension of each network output $D\in\rone^{d\times d}$ and $F\in\rone^{d}$. We denote the each network parameter as $\theta_D$ and $\theta_F$ and the output of each convolution encoder-decoder as $D^{nn}(f;\theta_D)$ and $F^{nn}(f;\theta_F)$. The Adam optimizer with learning late $10^{-3}$ is used with the learning rate scheduling.

\subsubsection{Train data for Step 1}To learn the operator $D(\cdot):f\mapsto D(f)$ and $F(\cdot):f\mapsto F(f)$, the samples of distribution $f$ are necessary as a training data, denoted as $\mathcal{D}$. In order to train the operators for distributions ranging from the initial condition to the equilibrium state, the training data is constructed with Gaussian functions and its variants. More precisely, the data $\mathcal{D}$ consists of the three type of distribution functions as follows:
\begin{itemize}
     \item (Gaussian function) $\{M_i(\textbf{v})\}_{i=1}^{100}$ where
     \begin{equation}\notag
         M_i(\textbf{v})\eqdef\frac{1}{(2\pi\sigma_i^2)^{d/2}}\exp\left(-\frac{|\textbf{v}-\textbf{c}_i|^2}{2\sigma_i^2}\right)
     \end{equation}with $\textbf{c}_i\in[-1,1]^d$ and $\sigma_i\in[0.8,1.2]$ 
     \item (Sum of two Gaussian functions) $\{M^2_i(\textbf{v})\}_{i=1}^{100}$ where
     \begin{equation}\notag
         M^2_i(\textbf{v})\eqdef M_{i_1}(\textbf{v})+M_{i_2}(\textbf{v}),
     \end{equation}where $M_{i_1}$ and $M_{i_2}$ are also random Gaussian functions.
     \item (Gaussian function with a small perturbation) $\{M^p_i(v)\}_{i=1}^{100}$ where
     \begin{equation}\notag
         M^p_i(\textbf{v})\eqdef M_i(\textbf{v})+M_i(v)g(\textbf{v}),
     \end{equation}where $g(\textbf{v})$ is a polynomial of degree 2 with randomly chosen coefficients in $[0,1]$. 
 \end{itemize}For the convenience, training data $\mathcal{D}$ is constructed by normalizing these 300 random distribution samples so that their volume becomes $0.2$. 
\subsubsection{Loss functions for Step 1}The supervised loss functions are used to train the operator surrogate models $D^{nn}(f;\theta_D)$ and $F^{nn}(f;\theta_F)$. More precisely, the loss functions are defined as follows:
\begin{equation}\label{loss_D}
    Loss_{D}(D^{nn})=\frac{1}{|\mathcal{D}|}\sum_{f\in\mathcal{D}}L\left(D^{nn}(f;\theta_D),D(f)\right)
\end{equation}
and
\begin{equation}\label{loss_F}
    Loss_{F}(F^{nn})=\frac{1}{|\mathcal{D}|}\sum_{f\in\mathcal{D}}L\left(F^{nn}(f;\theta_F),F(f)\right),
\end{equation}
where $L$ is a relative error with respect to Frobenius norm and $\mathcal{D}$ is a set of training data defined in the previous section. The target values $D(f)$ and $F(f)$ for each distribution $f\in\mathcal{D}$ are obtained using Gauss-Legendre quadrature, which is a numerical method for approximating the integral of function.

\subsection{Step 2: PINN to approximate the solution to the FPL equation}Using the pre-trained surrogate models $D^{nn}(f;\theta_D)$ and $F^{nn}(f;\theta_F)$, PINN is used to approximate the solution to the FPL equation during Step 2. A fully connected deep neural network (DNN) is used to parametrize the solution to the FPL equation. We denote the deep $L-$layer neural network (an input layer, $L-1$ hidden layers, and an output layer which is a $L-$th layer) parameters as follows:
\begin{itemize}
    \item $w^{(l+1)}$: the weight matrix between the $l-$th layer and the $(l+1)-$th layer of the network
    \item $b^{(l+1)}$: the bias vector between the $l-$th layer and the $(l+1)-$th layer of the network
    \item $m_l$: the width in the $l-$th layer of the network
    \item $\bar{\sigma}_l$: the activation function in the $l-$th layer of the network
\end{itemize}
where $m=(m_0,m_1,...,m_{L-1})$, $w=\{w^{(l)}\}_{l=1}^{L}$ and $b=\{b^{(l)}\}_{l=1}^{L}$. The details on the DNN structure and deep learning algorithm for PINN are precisely described in \cite{MR3881695,MR4313375}. In this study, the DNN takes a grid point $(t,\textbf{v})$ as input to obtain the output denoted by $f^{nn}(t,\textbf{v};m,w,b)$ that approximates the solution to the FPL equation. The DNN has 4 hidden layers which consist of 3(or 4)-100-100-100-100-1 neurons ($L=5$). The hyper-tangent activation function ($\bar{\sigma}_l=\frac{e^x-e^{-x}}{e^x+e^{-x}}$ for $l=1,...,L-1$) is used for hidden layers and the Softplus activation function ($\bar{\sigma}_L=\ln(1+e^x)$) is used only for the last layer. The Softplus function is used to preserve the positivity of the output $f^{nn}$, which is one of the main difficulties in numerical methods. Here, the Adam optimizer with learning late $10^{-3}$ is also used with the learning rate scheduling. Since the output $f^{nn}(t,\textbf{v};m,w,b)$ is continuous in time $t$ and continuous in velocity $\textbf{v}$, we obtain mesh-free and continuous-in-time solutions.

\subsubsection{Train data for Step 2} The data of grid points for each variable domain are necessary to train the neural network solution $f^{nn}(t,\textbf{v};m,w,b)$. The time interval $[0,T]$ is chosen as $[0,3]$ or $[0,5]$ depending on the initial conditions in the experiment. The interval is enough to reach the steady-state of FPL equation. Also, we truncate the velocity domain as
\begin{equation}\label{truncate}
    V\eqdef B_{R}(0)\subset\rone^d
\end{equation} with $R=5$ and assume that $f^{nn}$ is 0 on the outside of the domain $V$. The time grid is randomly chosen and the velocity grid is uniformly chosen in order to use the networks $D^{nn}(f;\theta_D)$ and $F^{nn}(f;\theta_F)$ consisting of convolutional layers with the uniform grid. The grid points for the governing equation in each epoch are chosen as
\begin{equation}\notag
    \{(t_i,\textbf{v}_j)\}_{i,j}\in[0,T]\times V
\end{equation}
with randomly chosen $t_i$ for $1\leq i\leq N_t$ and uniformly chosen $\textbf{v}_j$ for $1\leq j\leq N_v^d$. We set $N_t=10$, $N_v=64$ if $d=2$ and $N_v=32$ if $d=3$. For the initial condition, we use the grid points $\{(t=0,\textbf{v}_j)\}_j$.

\subsubsection{Loss functions for Step 2}
Using the pre-trained operator surrogate models $D^{nn}$ and $F^{nn}$, we define the following loss functions for Step 2 in our opPINN framework. For the governing equation \eqref{FPL}, we define the loss function as
\begin{multline}\label{loss_ge}
  Loss_{GE}(f^{nn}, D^{nn}, F^{nn}) \eqdef \int_{(0,T)}dt\int_{V}d\textbf{v} |\partial_t f^{nn}(t,\textbf{v};m,w,b)\\
  - \nabla_\textbf{v}\cdot(D^{nn}(f^{nn}(t,\textbf{v};m,w,b);\theta_D)\nabla f^{nn}(t,\textbf{v};m,w,b)\\
  -F^{nn}(f^{nn}(t,\textbf{v};m,w,b);\theta_F)f^{nn}(t,\textbf{v};m,w,b))|^2.
\end{multline}

We define the loss function for the initial condition as
\begin{equation}
Loss_{IC}(f^{nn}) \eqdef \int_{V} d\textbf{v}\left|f^{nn}(0,\textbf{v};m,w,b)-f_0(\textbf{v})\right|^2.
\end{equation}

Finally, we define the total loss as
\begin{equation}\label{loss_total}
  Loss_{Total}(f^{nn}, D^{nn}, F^{nn}) \eqdef Loss_{GE} + Loss_{IC}.
\end{equation}
Note that the integration in the loss functions are approximated by the Riemann sum on the grid points.

\section{On convergence of neural network solution to the FPL equation}\label{sec:convergence} In this section, we provide the theoretical proof on the convergence of the approximated neural network solution to the solution to the FPL equation \eqref{FPL}. In this section, we assume that the existence and the uniqueness of solutions to the FPL equation \eqref{FPL} are a priori given.

We first introduce the definition and the theorem from \cite{li1996simultaneous}.
\begin{definition}[Li, \cite{li1996simultaneous}]\label{C_hat} For a compact set $K$ of $\mathbb{R}^n$, we say $f\in \widehat{C}^{\xi}(K)$, $\xi\in \mathbb{Z}_+^n$ if there is an open $\Omega$ (depending on $f$) such that $K\subset \Omega$ and $f\in C^{\xi}(\Omega).$
\end{definition} 

\begin{theorem}[Li, Theorem 2.1, \cite{li1996simultaneous}]\label{Li_global} Let $K$ be a compact subset of $\mathbb{R}^n$, $n\ge 1$, and $f\in\widehat{C}^{\xi_1}(K)\cap\widehat{C}^{\xi_2}(K)\cap \cdots \widehat{C}^{\xi_q}(K)$, where $\xi_i \in \mathbb{Z}^n_+$ for $1\le i\le q$. Also, let $\bar{\sigma}$ be any non-polynomial function in $C^l(\mathbb{R})$, where $l=\max\{|\xi_i|:1\le i\le q\}$. Then for any $\epsilon>0,$ there is a network
$$f^{nn}(x)=\sum_{j=0}^\nu c_i\bar{\sigma}(\langle w_j,x\rangle +b_j), \ x\in \mathbb{R}^n,$$ where $c_i\in \mathbb{R},$ $w_j\in \mathbb{R}^n$, and $b_j\in \mathbb{R}$, $0\le j\le \nu$ such that 
$$\|D^kf-D^kf^{nn}\|_{L^{\infty}(K)}<\epsilon,$$
for $k\in \mathbb{Z}^n_+$, $k\le \xi_i$, for some $i$, $1\le i\le q.$
\end{theorem}

\begin{remark} $\nu$, $n$, and $q$ in Theorem \ref{Li_global} correspond to $m_1$, $d+1$, and $1$ in our deep $2-$layer neural network setting respectively. Also, $\xi_1$ corresponds to $(1,2,2)$ if $d=2$ and $(1,2,2,2)$ if $d=3$.
\end{remark}
\begin{remark} Although the neural network in Theorem \ref{Li_global} has only one hidden layer, we can generalize the result to the one with several hidden layers (see, \cite{hornik1989multilayer}). Therefore, from now, we may assume that the neural network architecture has only one hidden layer; i.e., $L=2$.
\end{remark}

First, we assume that the operators $D(f)$ and $F(f)$ are well approximated using the convolutional encoder-decoders $D^{nn}(f;\theta_D)$ and $F^{nn}(f;\theta_F)$. The convolutional encoder-decoder has been successfully applied for many applications. Many studies such as \cite{guo2016convolutional}, \cite{MR3800689} and \cite{MR3957452} proposed models based on the convolutional encoder-decoder, which show the possibility of learning function-to-function mapping. In this regard, our assumption on the operator surrogate models $D^{nn}$ and $F^{nn}$ are as follows.

\textbf{Assumption.} For any $\varepsilon_D,\varepsilon_F>0$, there exists $\theta_D$ and $\theta_F$ such that the operator surrogate model $D^{nn}(\cdot;\theta_D)$ satisfies
\begin{equation}\label{assum_D}
\|D^{nn}(\cdot;\theta_D) - D(\cdot)\|_{op} < \varepsilon_D
\end{equation}
and the operator surrogate model $F^{nn}(\cdot;\theta_F)$ satisfies
\begin{equation}\label{assum_F}
\|F^{nn}(\cdot;\theta_F) - F(\cdot)\|_{op} < \varepsilon_F\;\text{and} \|\nabla_{\textbf{v}}\cdot(F^{nn}(\cdot;\theta_F) - F(\cdot))\|_{op} < \varepsilon_F.
\end{equation}

\begin{remark}
    We update the network parameters $\theta_D$ and $\theta_F$ using the supervised loss functions \eqref{loss_D} and \eqref{loss_F}. We expect that we can find $\theta_D$ and $\theta_F$ so that the neural network models $D^{nn}$ and $F^{nn}$ satisfy \eqref{assum_D} and \eqref{assum_F} for small positive constants $\varepsilon_D$ and $\varepsilon_F$.
\end{remark}

Now, we provide two main theorems. First, we show that there exists a neural network parameter that can reduce the loss function as much as desired.

\begin{theorem}[Reducing the loss funtion]\label{thm1}Assume that the solution $f$ to \eqref{FPL} is in $\widehat{C}^{(1,2,2)}([0,T]\times V)$ if $d=2$ or $\widehat{C}^{(1,2,2,2)}([0,T]\times V)$ if $d=3$. Also, assume that the non-polynomial activation function $\bar{\sigma}(x)$ is in $C^5(\rone)$ if $d=2$ or $C^7(\rone)$ if $d=3$. Then, for any $\varepsilon>0$, there exists $m$, $w$, $b$, $\theta_D$ and $\theta_F$ such that the neural network solution $f^{nn}(t,\textbf{v};m,w,b)$, $D^{nn}(f;\theta_D)$ and $F^{nn}(f;\theta_F)$ satisfy
\begin{equation}
\|Loss_{Total}(f^{nn}, D^{nn}, F^{nn})\|_{L^\infty} < \varepsilon.
\end{equation}
\end{theorem}
\begin{proof}
This is similar to that of Theorem 3.4 of \cite{MR4116803} using Theorem \ref{Li_global}.
\end{proof}

We introduce our second main theorem which shows that the neural network solution $f^{nn}(t,\textbf{v};m,w,b)$ converges to a priori classical solution to the FPL equation when the total loss function\eqref{loss_total} is reduced. We use the following two properties based on the theoretical backgrounds on the FPL equation.
\begin{itemize}
    \item Based on the result of the propagation of Gaussian bounds on FPL equation from \cite{MR3778645}, we have
\begin{equation}\label{assum_decay}
f(t,\textbf{v})\leq G_1e^{-\alpha|\textbf{v}|^2},
\end{equation}for some constant $\alpha>0$, if we suppose that the initial data $f_0$ is controlled by $G_0e^{-\alpha|v|^2}$, for some constants $A_0>0$ and a sufficiently small $\alpha>0$ where $A_1$ depends on  $A_0$, $\alpha$ and $d$.
    \item From Proposition 4 in \cite{MR1737547}, we have an ellipticity property for the matrix $D(f)$. If we assume that the solution $f$ of \eqref{FPL} satisfies $f\in\widehat{C}^{(1,2,2)}([0,T]\times V)$ if $d=2$ or $f\in\widehat{C}^{(1,2,2,2)}([0,T]\times V)$ if $d=3$ and that the mass, energy and entropy of the solution is bounded above, then there exists a constant $K>0$ such that the matrix $D(f)$ satisfies
    \begin{equation}\label{assum_ellipticity}
        \forall \xi\in\mathbb{R}^d, \quad \sum_{1\leq i,j\leq d}\xi_i[D(f)]_{ij}\xi_j\geq K(1+|v|^\gamma)|\xi|^2,
    \end{equation}
    where the constant $K$ depending only on $\gamma$ and upper bounds of mass, energy and entropy of the solution $f$.
\end{itemize}
We also assume that our compact domain $V=B_R(0)\subset\mathbb{R}^d$ defined in \eqref{truncate} is chosen sufficiently large so that we can extend $f^{nn}(t,\textbf{v};m,w,b)$ as 0 on the outside of the compact domain $V$, i.e.
    \begin{equation}\label{assum_nnzero}
        f^{nn}(t,\textbf{v};m,w,b)=0,\quad\text{if }\textbf{v}\in\rone^d\setminus V.
    \end{equation}Also, we have
\begin{equation}\label{assum_bdry}
\left|\partial^{k}_{\textbf{v}}f(t,\textbf{v})-\partial^{k}_{\textbf{v}}f^{nn}(t,\textbf{v};m,w,b)\right|_{\textbf{v}\in \partial V\cup(\rone^d\setminus V)} < \epsilon,
\end{equation}for some sufficiently small $\epsilon>0$ and $|k|\leq2$.

Now, we introduce our second main theorem.
\begin{theorem}[Convergence of the neural network solution]\label{thm2} Assume that $f$ is a solution to \eqref{FPL} which belongs to $\widehat{C}^{(1,2,2)}([0,T]\times V)$ if $d=2$ or $\widehat{C}^{(1,2,2,2)}([0,T]\times V)$ if $d=3$, and the index $\gamma$ satisfies $\gamma>\max\{-\frac{d}{2}-2,-d-1\}$. If the solution $f$ and the neural network solution $f^{nn}(t,\textbf{v};m,w,b)$ satisfy \eqref{assum_decay}, \eqref{assum_nnzero} and \eqref{assum_bdry}, then it holds that 
\begin{equation}
\left\| f^{nn}(\cdot,\cdot;m,w,b) - f \right\|_{L^\infty_t([0,T];L^{2}_\textbf{v}(V))}\leq C{Loss}_{Total}(f^{nn}, D^{nn}, F^{nn})+\varepsilon_D+\varepsilon_F+\epsilon,
\end{equation}
where $C$ is a positive constant depending on $T$, $R$ and $d$ $(d=2,3)$.
\end{theorem}
\begin{proof}
The proof of this theorem is provided in Appendix \ref{appendix:proof}.
\end{proof}
\begin{remark}[Connection between Theorem \ref{thm1} and Theorem \ref{thm2}]
    Theorem \ref{thm1} shows the existence of the neural network parameters $m$, $w$, $b$, $\theta_D$ and $\theta_F$, which make the total loss \eqref{loss_total} as small as we want. If we want to approximate the neural network solution $f^{nn}$ close to a priori solution $f$, the deep learning algorithm (PINN) is used to find the neural network parameters which make the total loss small so that the error $\|f^{nn}-f\|$ is reduced based on Theorem \ref{thm2}.
\end{remark}
\section{Simulation Results}\label{sec:simulation}
We next present the simulation results of the opPINN framework in two subsections. We first show the results on the operator surrogate models $D^{nn}(f;\theta_D)$ and $F^{nn}(f;\theta_F)$. It contains the results on the surrogate model combined with the traditional numerical methods to approximate the collision term $Q(f,f)$. The following subsection shows the results on the neural network solution to the FPL equation obtained from the opPINN framework. We compare our method with the fast spectral method from \cite{MR1795398}, which is well-known methods for simulating the FPL equation. The spectral method is implemented in Python using NumPy library.

\begin{figure}[H]
  \includegraphics[width=\textwidth, draft=false]{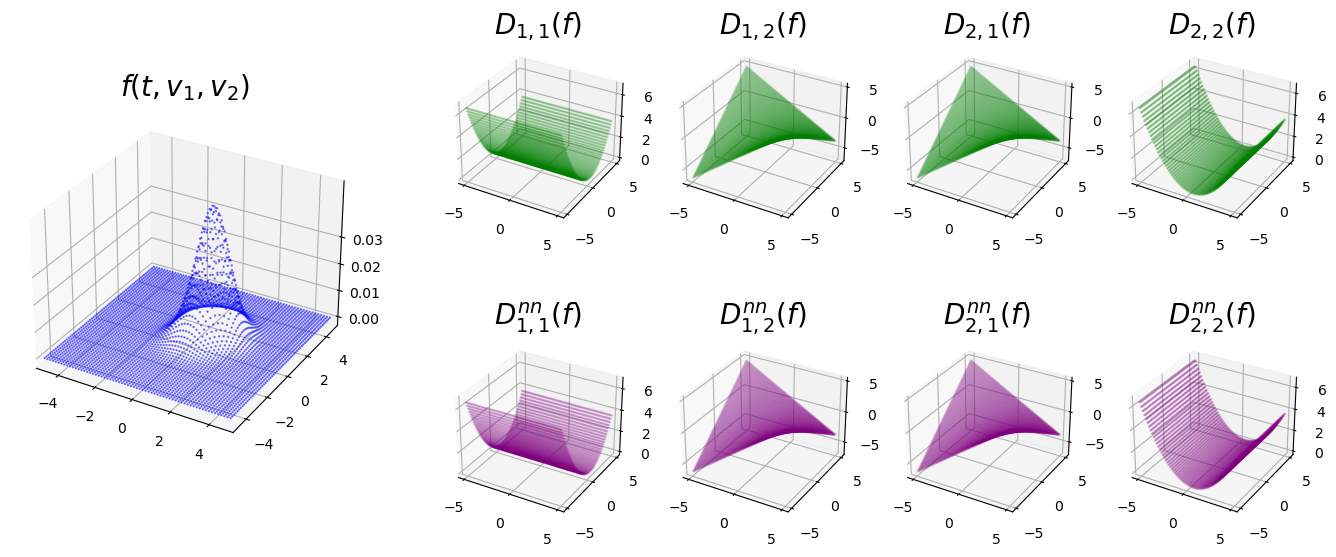}
  \caption{Examples of inputs (blue) and outputs (purple) of the operator surrogate model $D^{nn}(f;\theta_D)$ for distribution $f$ which is not in train data. The upper row plots are the target values $D(f)$ (green). Note that $D(f)\in\rone^{2\times2}$ when $d=2$.}
  \label{fig:operator_D}
\end{figure}

\begin{figure}[H]
  \includegraphics[width=\textwidth, draft=false]{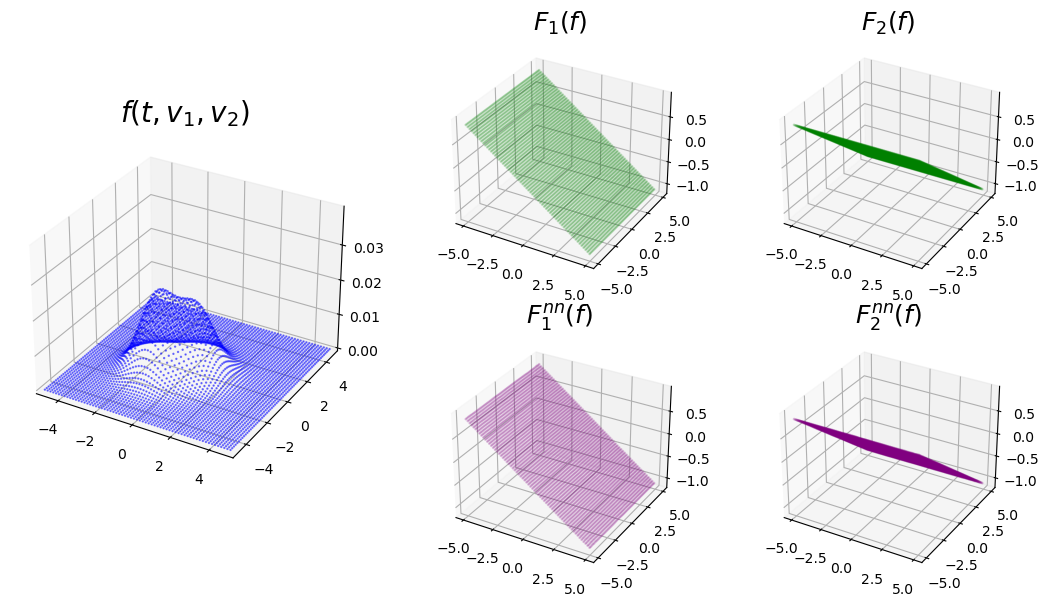}
  \caption{Examples of inputs (blue) and outputs (purple) of the operator surrogate model $F^{nn}(f;\theta_F)$ for distribution $f$ which is not in train data. The upper row plots are the target values $F(f)$ (green). Note that $F(f)\in\rone^2$ when $d=2$.}
  \label{fig:operator_F}
\end{figure}

\subsection{Results on Step 1: the operator surrogate models}
In this section, we focus on the results of the operator learning to approximate the two terms $D(f)$ in \eqref{Df} and $F(f)$ in \eqref{Ff} during Step 1. As explained in Section \ref{subsec:step1}, the two convolution encoder-decoders are used as the operator surrogate model to approximate the each operator $D(f)$ and $F(f)$. Figure \ref{fig:operator_D} shows the target value $D(f)$ and the output of the trained operator surrogate model $D^{nn}(f;\theta_D)$ for the random distribution function $f$ which is not in train data. Here, the figure shows the two-dimensional case with $N=64$. It shows that the output of the operator surrogate model is well matched to the target value for each of the four components of the matrix $D(f)$. The average of relative $L^2$ errors for the 100 test samples are less than 0.006. Figure \ref{fig:operator_F} shows the result on the operator surrogate model $F^{nn}(f;\theta_F)$ with $N=64$. It also shows that the target value $F(f)$ is well approximated where the average of relative $L^2$ errors for 100 test samples is 0.008.

\begin{figure}[H]
  \includegraphics[width=\textwidth, draft=false]{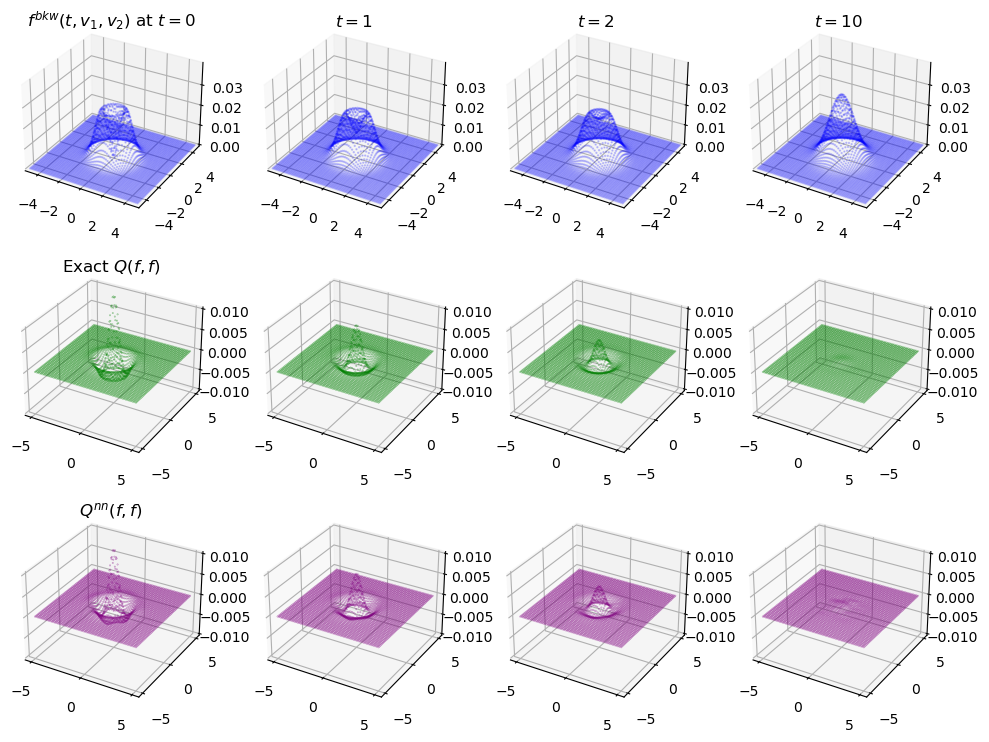}
  \caption{The approximated Landau collision operator $Q^{nn}(f,f)$ (purple in third row) using the operator surrogate models $D^{nn}(f;\theta_D)$ and $F^{nn}(f;\theta_F)$ combined with the FDM for the BKW solutions $f=f^{bkw}$. The first row shows the BKW solutions $f^{bkw}$ at time $t=0,1,2,10$ (blue) and the second row shows the exact value of the Landau collision operator $Q(f^{bkw},f^{bkw})$ (green).}
  \label{fig:operator_Q}
\end{figure}

Using the two operator surrogate models, the solution to the FPL equation is approximated using PINN as described in the following subsection (Section \ref{subsec:step2}) which shows the results of the proposed opPINN framework. However, the operator surrogate models can also be combined with the traditional numerical schemes to approximate the FPL equation. It is possible to directly approximate the Landau collision operator $Q(f,f)$ by calculating the derivative of distribution function $f$ using the FDM. Figure \ref{fig:operator_Q} shows the direct approximation of collision term $Q(f,f)$ using the operator surrogate models $D^{nn}(f;\theta_D)$ and $F^{nn}(f;\theta_F)$ with $N=64$ combined with the FDM to approximate the derivative of distribution function $f$. To check that the collision term $Q(f,f)$ is well approximated, the sample of the function $f$ and the corresponding collision term are chosen as an exact solution of the FPL equation, called BKW solution. The BKW solution is one of well-known exact solutions to the FPL equation (See appendix in \cite{MR4121055}). To make the volume of BKW solution as 0.2, we choose the normalized BKW solution that is given by
\begin{equation}\label{bkw}
    f^{\text{bkw}}(t,\textbf{v})=\frac{1}{10\pi K^2}\exp\left(-\frac{|\textbf{v}|^2}{2K}\right)\left(2K-1+\frac{1-K}{2K}|\textbf{v}|^2\right),\; K=1-\frac{1}{2}\exp\left(-\frac{t}{8}\right),
\end{equation}
which $f^{\text{bkw}}(t,\textbf{v})$ has volume 0.2 when the constant of collision kernel \eqref{kernel} is $\Lambda=\frac{5}{16}$. The corresponding collision operator $Q(f^{\text{bkw}},f^{\text{bkw}})$ can also be explicitly calculated. Figure \ref{fig:operator_Q} shows the exact value of $Q(f,f)$ and approximated $Q^{nn}(f,f)$ for the solutions $f=f^{\text{bkw}}(t,\textbf{v})$ at $t=0,1,2,10$. It shows that the Landau collision term $Q(f,f)$ is well approximated using the operator surrogate models combined with the FDM.

We also expect that the method using the proposed operator surrogate models combined with the FDM is faster than the traditional numerical methods for approximating the Landau collision operator. The reason is that the operator surrogate models only need an inference time after the models are trained using the train data. To confirm this, we compare the computational time to approximate collision operator $Q(f,f)$ using our operator surrogate model with the spectral method in different modes of velocity. For the spectral method, the FPL kernel modes are pre-calculated (see \cite{MR1795398} for details) so that the computational time for the pre-calculation is excluded in the experiment. The proposed method has a larger computational time when the modes of velocity is small as shown in Figure \ref{fig:operator_speed}. However, the spectral method dramatically increases the computational time as the mode becomes larger. It is well known that the computational complexity of the spectral method is known as $O(N\log N)$. On the other hand, the computational cost for the proposed method increased more slowly as the mode becomes larger. It is expected that the time gap of the two methods becomes larger as the number of velocity modes are increasing. Therefore, the proposed method may have great advantages when a fast computation is required in a large modes of velocity.

\begin{figure}[H]
  \includegraphics[width=\textwidth, draft=false]{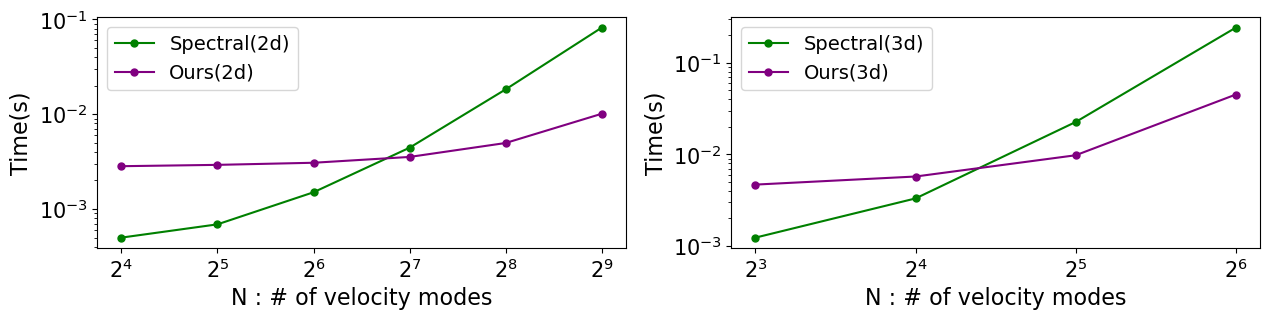}
  \caption{The comparison of the computational time for approximating the collision operator 100 times as the number of the velocity modes $N$ varies. The proposed operator surrogate model with FDM (purple) is compared with the spectral method (green) for two-dimensional (left) and three-dimensional (right) FPL equation.}
  \label{fig:operator_speed}
\end{figure}

Using the approximated collision operator $Q^{nn}(f,f)$, the FPL equation also can be solved by combining it with the Euler method. We compare the proposed method with the spectral method to simulate the FPL equation with the BKW initial condition. The approximated solutions using the two methods and the exact value of BKW solution at time $t=0$ to $3$ is compared in Figure \ref{fig:operator_euler} when $N=64$ and $N=128$. It shows that the proposed method is less accurate than the spectral method when the number of velocity modes are $N=64$. The error is generated from the FDM and is accumulated over time steps. In the case of $N=128$, the proposed method has a similar accuracy with the spectral method as shown in Figure \ref{fig:operator_euler}.

\begin{figure}[H]
  \includegraphics[width=1\textwidth, draft=false]{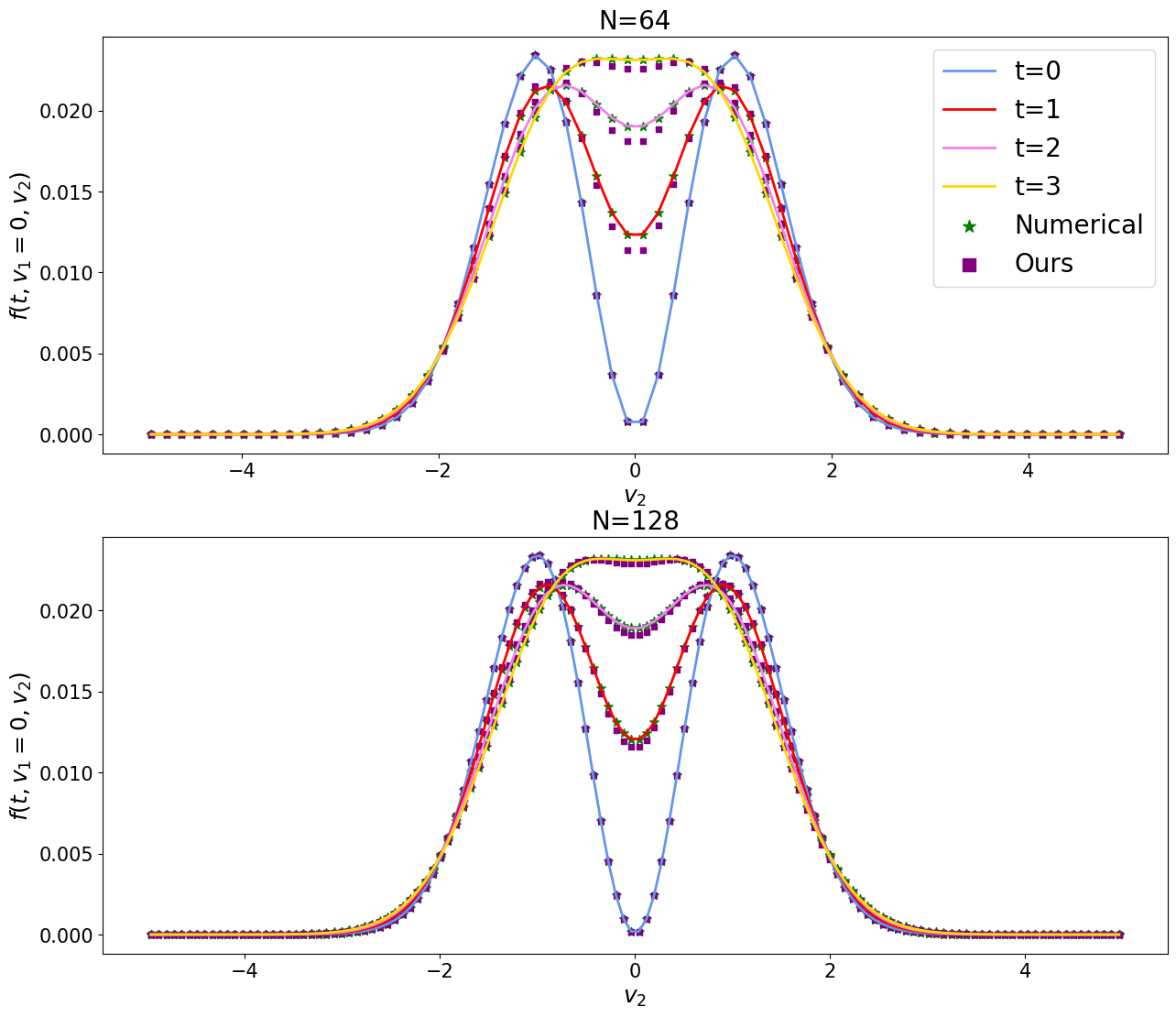}
  \caption{The approximated solution to the two-dimensional FPL equation using the spectral method with Euler method (green points) and the operator surrogate model combined with the FDM and the Euler method (purple points) with $N=64$ (upper) and $N=128$ (lower). The exact values of BKW solution $f^{bkw}$ at $t=0,1,2,3$ are shown in different colors for reference.}
  \label{fig:operator_euler}
\end{figure}

\subsection{Results on Step 2: the neural network solution to the FPL equation using the opPINN framework}\label{subsec:step2}We next show the neural network solution to the FPL equation using the proposed opPINN framework. The trained operator surrogate models $D^{nn}(f;\theta_D)$ and $F^{nn}(f;\theta_F)$ with $N=64$ during Step 1 are used in this section. We test the opPINN framework for the FPL equation by varying the initial conditions, collision types, and velocity dimension. In each experiment, we provide the pointwise values of our neural network solution $f^{nn}(t,\textbf{v};m,w,b)$ and the physical quantities of the neural network solution, e.g. mass $\text{Mass}(t)$, momentum $\text{Mom}(t)$, kinetic energy $\text{KE}(t)$ (Eq. \eqref{physical_quantities}) and entropy $\text{Ent}(t)$ (Eq. \eqref{entropy}). We approximate the integral in these physical quantities using Gauss-Legendre quadrature on the grid points. We compare our results with the existing theoretical results as explained in Section \ref{properties}. Furthermore, we observe the long-time behavior of the neural network solution $f^{nn}$ to check the convergence of the neural network solution to Maxwellian given in \eqref{maxwellian}.

\subsubsection{Results by varying the initial conditions} First, we impose Maxwellian initial condition given by
\begin{equation}\label{ini_onegauss}
    f(0,\textbf{v})=0.2\times\frac{1}{2\pi}\exp\left(-\frac{|\textbf{v}|^2}{2}\right)
\end{equation}with $\Lambda=1$ for the collision kernel. Since the initial condition is already steady-state, we expect the neural network solution $f^{nn}(t,v_1,v_2;m,w,b)$ keeps the initial condition as the time increases. Figure \ref{fig:opPINN_ini_onegauss} shows the pointwise values of $f^{nn}$ as time $t$ varies. Note that the first row in the figure shows the 3-d plots of the distribution $f(t,v_1,v_2)$ and the second row shows its cross section with $v_1=0$ for visualization purposes. As we expect, the neural network solution does not change from $t=0$ to $t=5$. 

\begin{figure}[H]
  \includegraphics[width=\textwidth, draft=false]{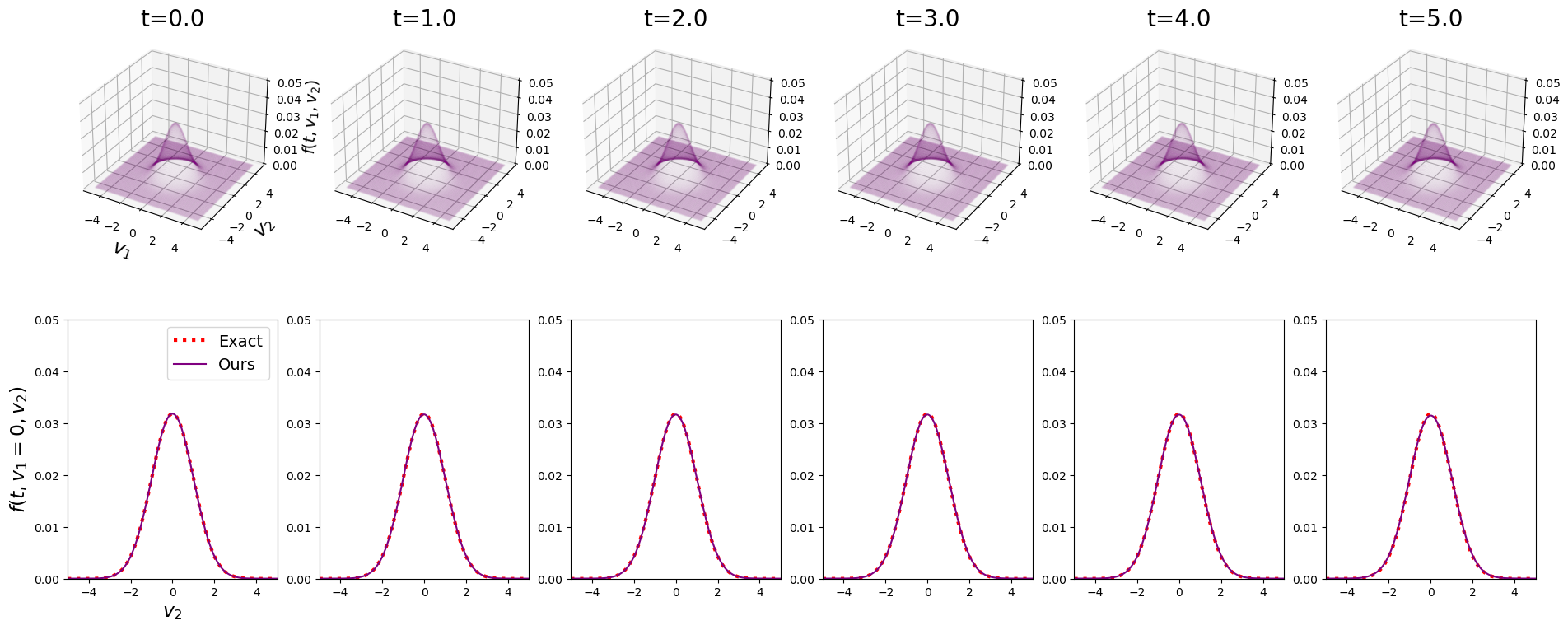}
  \caption{The pointwise values of $f^{nn}(t,v_1,v_2;m,w,b)$ as $t$ varies for Maxwellian initial condition \eqref{ini_onegauss}. Note that the upper row shows the 3-d plots of the distribution $f(t,v_1,v_2)$ and the lower row shows its cross section with $v_1=0$ for visualization purposes. The values of the exact solutions are given via the red-dotted lines. Note that the initial condition is the exact solution at all times in this case.}
  \label{fig:opPINN_ini_onegauss}
\end{figure}

Second, we use BKW solution at time $t=0$ as the initial condition given by
\begin{equation}\label{ini_bkw}
    f(0,\textbf{v})=f^{\text{bkw}}(0,\textbf{v}),
\end{equation}as explained in \eqref{bkw} with $\Lambda=\frac{5}{16}$ for the collision kernel. It is one of well-known exact solutions to the FPL equation, which is used to test the numerical methods in many numerical studies \cite{MR4121055,carrillo2021random,MR1795398}. Figure \ref{fig:opPINN_ini_bkw} shows the pointwise values of $f^{nn}(t,v_1,v_2;m,w,b)$ from $t=0$ to $t=5$. We compare the neural network solution with the exact solution $f^{ext}$ that is plotted in red dotted line. It shows the neural network solution is well approximated. Also, Figure \ref{fig:opPINN_ini_bkw_physical} shows the physical quantities for the neural network solution. The red dotted lines show each initial physical quantity of the initial distribution. It shows the neural network solution obtained by the proposed method satisfies the conservation of the mass, momentum and kinetic energy. In particular, the entropy of the neural network solution is a non-increasing function which is consistent to the analytical property given in \eqref{entropy_nonincrease}.

\begin{figure}[H]
  \includegraphics[width=\textwidth, draft=false]{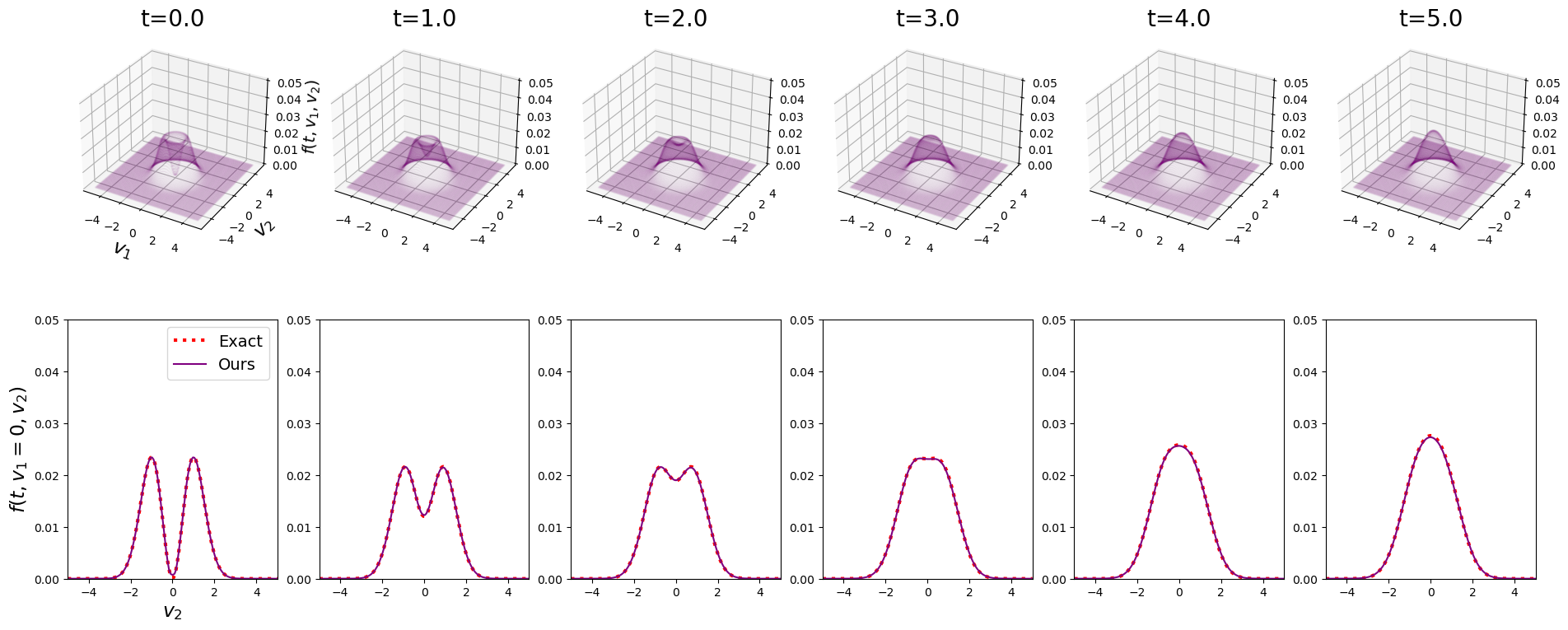}
  \caption{The pointwise values of $f^{nn}(t,v_1,v_2;m,w,b)$ as $t$ varies for BKW initial condition \eqref{ini_bkw}. Note that the upper row shows the 3-d plots of the distribution $f(t,v_1,v_2)$ and the lower row shows its cross section with $v_1=0$ for visualization purposes. The exact values of the BKW solution are given via the red-dotted lines.}
  \label{fig:opPINN_ini_bkw}
\end{figure}

\begin{figure}[H]
  \includegraphics[width=\textwidth, draft=false]{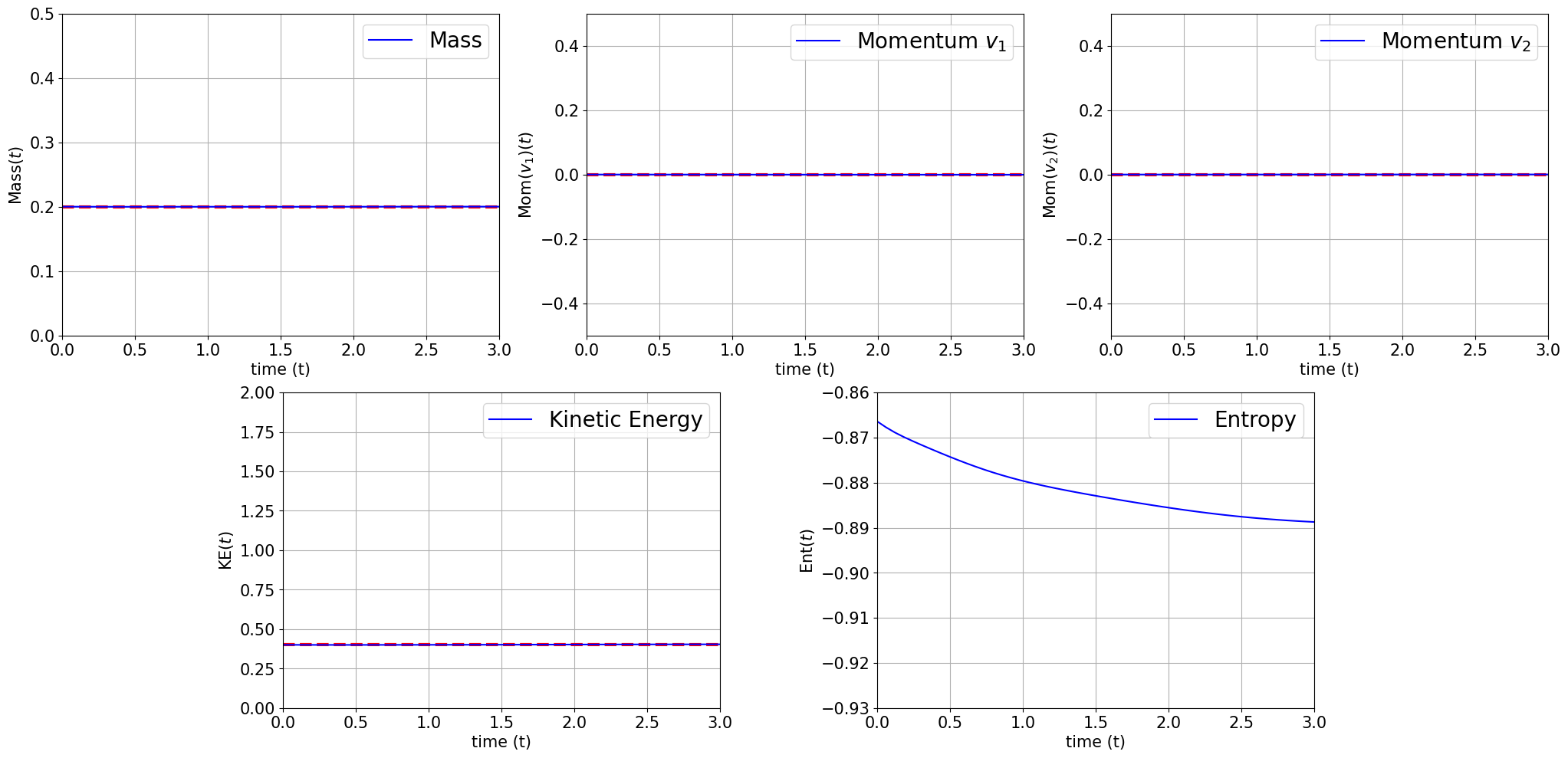}
  \caption{The physical quantities of $f^{nn}(t,v_1,v_2;m,w,b)$ for BKW initial condition \eqref{ini_bkw}. The initial values of the physical quantities are given via the red-dotted line. Note that the Entropy is a non-increasing function.}
  \label{fig:opPINN_ini_bkw_physical}
\end{figure}

\begin{figure}[H]
  \includegraphics[width=\textwidth, draft=false]{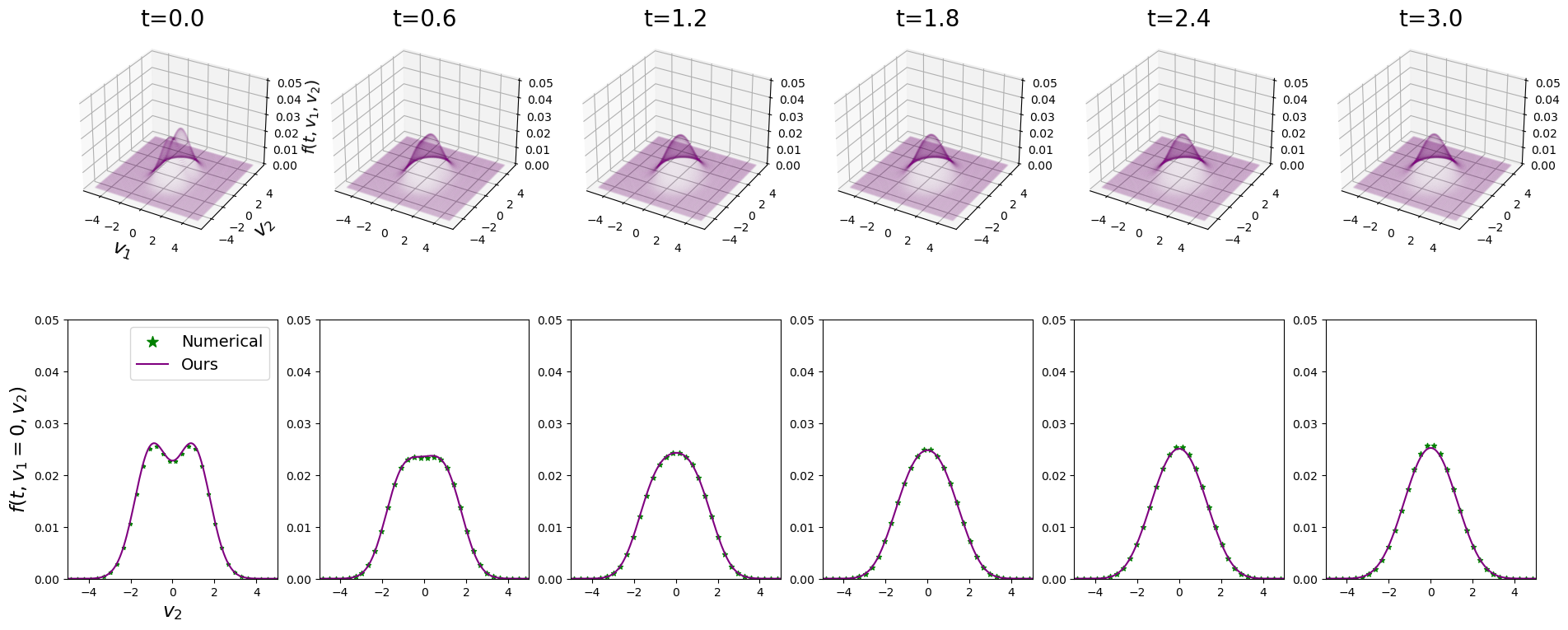}
  \caption{The pointwise values of $f^{nn}(t,v_1,v_2;m,w,b)$ as $t$ varies for two-Gaussian initial condition \eqref{ini_col}. Note that the upper row shows the 3-d plots of the distribution $f(t,v_1,v_2)$ and the lower row shows its cross section with $v_1=0$ for visualization purposes. The values of numerical solutions (spectral method) are given via the green star points.}
  \label{fig:opPINN_col_pointwise}
\end{figure}

\begin{figure}[H]
  \includegraphics[width=\textwidth, draft=false]{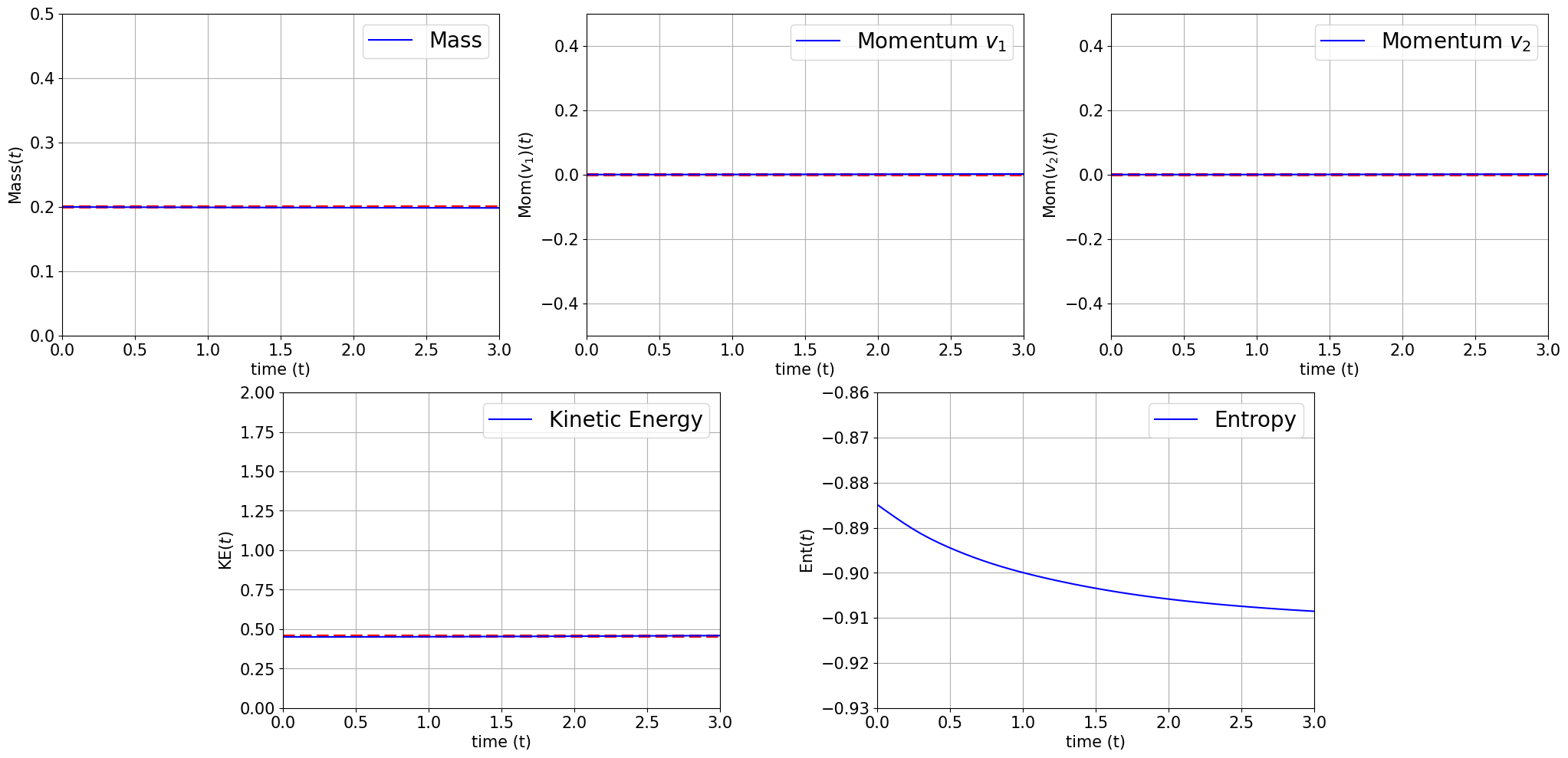}
  \caption{The physical quantities of $f^{nn}(t,v_1,v_2;m,w,b)$ for two-Gaussian initial condition \eqref{ini_col}. The initial values of the physical quantities are given via the red-dotted line. Note that the Entropy is a non-increasing function.}
  \label{fig:opPINN_col_physical}
\end{figure}

\subsubsection{Results for Columbian case $(\gamma=-3)$} We also change the index of power $\gamma$ in collision kernel \eqref{kernel} to test the proposed method for other collision types. We set the $\gamma=-3$ that is the Columbian potential case. We impose two-Gaussian initial condition for the initial condition given by
\begin{equation}\label{ini_col}
    f(0,\textbf{v})=\frac{1}{C}\left(\frac{1}{2\pi\sigma^2}\exp\left(-\frac{|\textbf{v}-c_1|^2}{2\sigma^2}\right)+\frac{1}{2\pi\sigma^2}\exp\left(-\frac{|\textbf{v}-c_2|^2}{2\sigma^2}\right)\right)
\end{equation}where $c_1=(0,1)$, $c_2=(0,-1)$, $\sigma=0.8$ and $C$ is the normalized constant to make the volume 0.2 with $\Lambda=5$ for the collision kernel. The initial condition is used to test the numerical methods in many numerical studies \cite{MR4108209,MR4121055,MR1795398} Since we do not have the exact solution in this case, the numerical solutions obtained by the spectral method are used for comparison. The spectral method with $N=32$ combined with the Euler method ($\Delta t=0.001$) is used to approximate the FPL equation.

Figure \ref{fig:opPINN_col_pointwise} shows the pointwise values of $f^{nn}(t,v_1,v_2;m,w,b)$ from $t=0$ to $t=3$. The neural network solution is well matched to the spectral method that is plotted in green star points. It shows the proposed method also works well for the Columbian potential case. In addition, to ensure that the neural network solutions are well approximated, we also check the physical quantities over time. Figure \ref{fig:opPINN_col_physical} shows the mass, the momentum, the kinetic energy and the entropy of the neural network solution. As we expect, the mass, the momentum for each velocity direction, and kinetic energy are preserved as time varies. Also, the entropy of the neural network solution is a non-increasing function.

\begin{figure}[H]
  \includegraphics[width=\textwidth, draft=false]{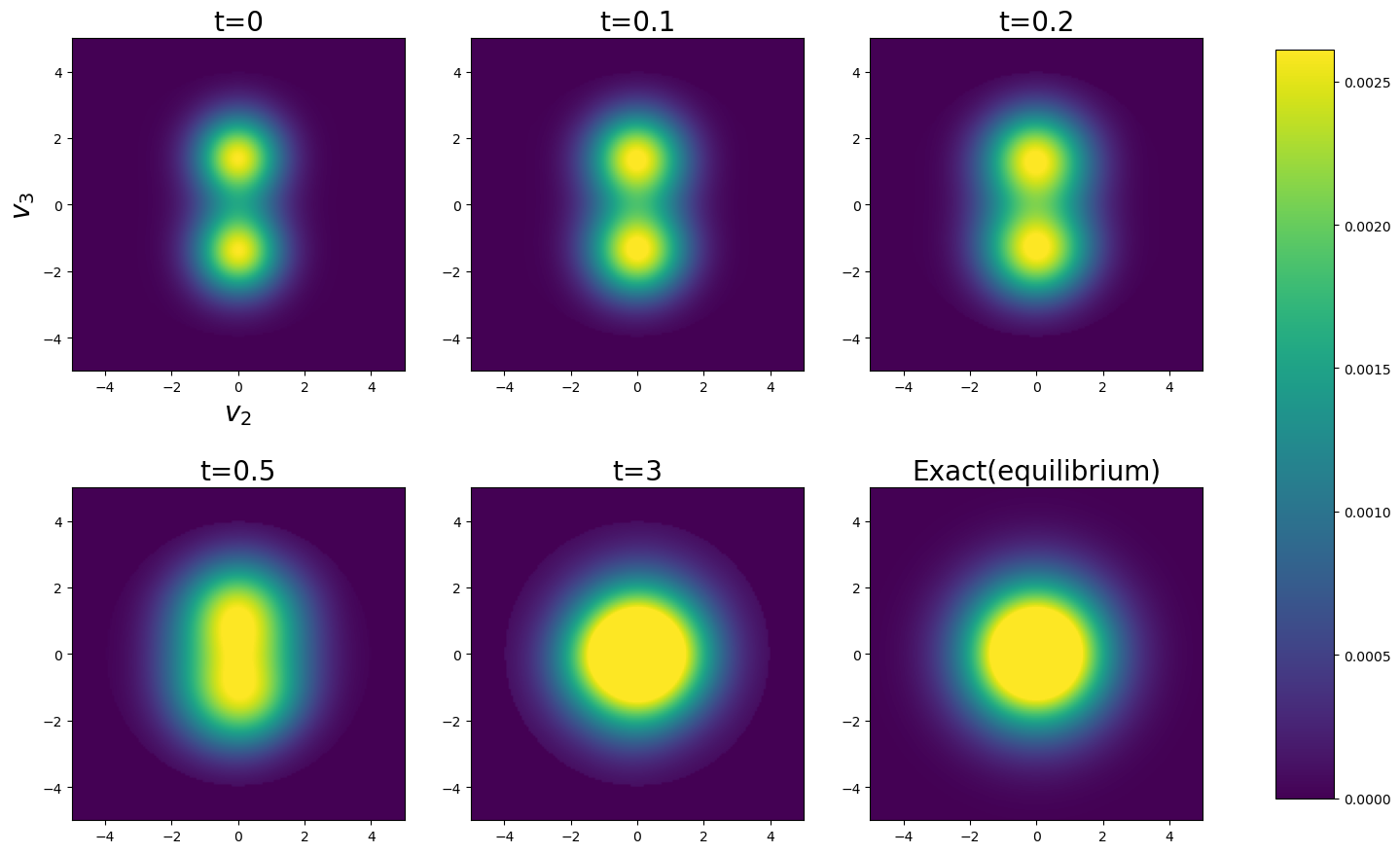}
  \caption{The pointwise values of $f^{nn}(t,v_1,v_2,v_3;m,w,b)$ as $t$ varies for two-Gaussian initial condition \eqref{ini_3d}. Note that figure shows the cross section of $f^{nn}(t,v_1,v_2,v_3;m,w,b)$ with $v_1=0$ for visualization purposes. The right plot in the lower row shows the steady-state values of the FPL equation.}
  \label{fig:opPINN_3d_pointwise}
\end{figure}

\subsubsection{Results for 3-dimensional FPL equation $(d=3)$} Furthermore, we apply the proposed opPINN framework for the 3-dimensional FPL equation. Similar to other studies \cite{filbet2020spectral}, we focus on the test for the convergence of the neural network solution to the equilibrium, a 3-dimensional Maxwellian function. We impose the two-Gaussian initial condition
\begin{equation}\label{ini_3d}
    f(0,\textbf{v})=\frac{1}{C}\left(\frac{1}{(2\pi\sigma^2)^\frac{3}{2}}\exp\left(-\frac{|\textbf{v}-c_1|^2}{2\sigma^2}\right)+\frac{1}{(2\pi\sigma^2)^\frac{3}{2}}\exp\left(-\frac{|\textbf{v}-c_2|^2}{2\sigma^2}\right)\right)
\end{equation}where $c_1=(1.4,1.4,0)$, $c_2=(-1.4,-1.4,0)$, $\sigma=0.9$ and $C$ is the normalized constant to make the volume 0.2 with $\Lambda=1$ for the collision kernel. Figure \ref{fig:opPINN_3d_pointwise} shows pointwise values of $f^{nn}(t,v_1,v_2,v_3;m,w,b)$ as time $t$ varies and the analytic equilibrium. For the visualization purpose, the cross sections of the neural network solution $f^{nn}(t,v_1,v_2,v_3;m,w,b)$ with $v_1=0$ are plotted in two-dimensional images. It shows that the neural network solution at $t=3$ converges to the equilibrium.

\section{Conclusion}\label{sec:conclusion}
This paper proposes the opPINN framework to approximate the solution to the FPL equation. The opPINN framework is a hybrid method using the operator learning approach with PINN. It provides a mesh-free neural network solution by effectively approximating the complex Landau collision operator of the equation. The experiments on the FPL equation under the various setting show that the neural network solution satisfies the known analytic properties on the equation. Furthermore, we show that the approximated neural network solutions converge to the classical solutions of the FPL equation based on the theoretical properties of the FPL equation.

Other models for operator learning are developed lately including DeepONet \cite{lu2019deeponet} and Neural Operator \cite{li2020neural}. Although the convolution encoder-decoder networks are used for the operator surrogate model, other structures for operator learning also can be used in our framework to improve performances. The non-homogeneous FPL equation and the FPL equation on the bounded domain are possible extensions of our work. Boltzmann equation has a more complex collision term than the FPL equation. For future work, we expect that the proposed framework can be applied to the Boltzmann equation for various collision kernels.

\section*{Acknowledgement}
H. J. Hwang was supported by the National Research Foundation of Korea (NRF) grant funded by the Korea government (MSIT) (NRF-2017R1E1A1A03070105, NRF-2019R1A5A1028324) and by Institute for Information \& Communications Technology Promotion (IITP) grant funded by the Korea government(MSIP) (No.2019-0-01906, Artificial Intelligence Graduate School Program (POSTECH)). J. Y. Lee was supported by a KIAS Individual Grant (AP086901) via the Center for AI and Natural Sciences at Korea Institute for Advanced Study and by the Center for Advanced Computation at Korea Institute for Advanced Study.

\appendix
\section{Proof of Theorem \ref{thm2}}\label{appendix:proof}
\begin{proof}[Proof of Theorem \ref{thm2}]For convenience, assume $\Lambda=1$ in \eqref{kernel}. We define the residual error of the neural network output $f^{nn}$ as follows:
$$d_{ge}(t,\textbf{v})\eqdef -\partial_tf^{nn}(t,\textbf{v})+\nabla_\textbf{v}\cdot\left(D^{nn}(f^{nn})\nabla f^{nn}-F^{nn}(f^{nn})f^{nn}\right).$$
Then, we consider the following equation on the difference between $f$ and $f^{nn}$ as
\begin{multline}\label{ge_diff}
\partial_t(f-f^{nn})-\nabla_\textbf{v}\cdot\left(\left(D(f)\nabla f-D^{nn}(f^{nn})\nabla f^{nn}\right)-\left(F(f)f-F^{nn}(f^{nn})f^{nn}\right)\right)\\
=d_{ge}(t,\textbf{v}) 
\end{multline}
for $(t,\textbf{v})\in[0,T]\times V$. We derive the energy identity by multiplying $(f-f^{nn})$ onto \eqref{ge_diff} and integrating it over $V$ as
\begin{multline}\label{energy_eq}
\frac{1}{2}\frac{d}{dt}\|(f-f^{nn})(t,\cdot)\|^2_{L^{2}_v(V)}-\underbrace{\int_Vd\textbf{v}(f-f^{nn})\nabla_\textbf{v}\cdot\left(D(f)\nabla f-D^{nn}(f^{nn})\nabla f^{nn}\right)}_{A(t)}\\+\underbrace{\int_Vd\textbf{v}(f-f^{nn})\nabla_\textbf{v}\cdot\left(F(f) f-F^{nn}(f^{nn}) f^{nn}\right)}_{B(t)}=\int_Vd\textbf{v}(f-f^{nn})d_{ge}(t,\textbf{v}).
\end{multline}
For the right-hand side of \eqref{energy_eq}, we note that
\begin{equation}
\int_Vd\textbf{v}(f-f^{nn})d_{ge}(t,\textbf{v})\leq\frac{1}{2}\|f-f^{nn}\|_{L^2_\textbf{v}( V)}^2 + \frac{1}{2}\|d_{ge}(t,\textbf{v})\|_{L^2_\textbf{v}( V)}^2.
\end{equation}
Also, note that
\begin{align}\label{At_all}
A(t)=&\int_Vd\textbf{v}(f-f^{nn})\nabla_\textbf{v}\cdot\left[(D(f)-D^{nn}(f^{nn}))\nabla f^{nn}+D(f)\nabla_\textbf{v}(f-f^{nn})\right]\\
=&\underbrace{\int_{\partial V}dS(f-f^{nn})[(D(f)-D^{nn}(f^{nn}))\nabla f^{nn}+D(f)\nabla_\textbf{v}(f-f^{nn})]\cdot \hat{n}}_{A_1(t)}\notag\\
&-\underbrace{\int_Vd\textbf{v}\nabla_\textbf{v}(f-f^{nn})\cdot[(D(f)-D^{nn}(f^{nn}))\nabla f^{nn}+D(f)\nabla_\textbf{v}(f-f^{nn})]}_{A_2(t)}\notag
\end{align}
by the divergence theorem where $\hat{n}$ is the outward unit normal vector on the boundary $\partial V$. Since $f^{nn}\in C^{(1,2,2)}$ or $C^{(1,2,2,2)}$ on the compact domain $[0,T]\times V$, $\nabla f^{nn}$ is bounded. We also have the boundedness of each entry of $D(f)$ since the function $D(f)(\textbf{v})$ is continuous function with respect to $\textbf{v}$ on the compact domain $V$. Using the assumption \eqref{assum_bdry}, we deduce that
\begin{equation}
A_1(t)\leq M_1\epsilon\int_{\partial V}dS \sum_{1\leq i,j\leq d}|(D(f)-D^{nn}(f^{nn}))_{ij}| + M_2\epsilon
\end{equation}
for some positive constant $M_1$ and $M_2$. Note that
\begin{align}\label{D_D}
    |(D(f)-D^{nn}(f^{nn}))_{ij}|=&|(D(f)-D(f^{nn}))_{ij}+(D(f^{nn})-D^{nn}(f^{nn}))_{ij}|\\
    \leq& \left|\int_{\rone^d}d\textbf{v}_*\Phi(\textbf{v}_*)_{ij}(f-f^{nn})(\textbf{v}-\textbf{v}_*)\right|+\|D-D^{nn}\|_{op}\notag.
\end{align}
We can bound the first integration term in the right side of \eqref{D_D} as
\begin{align*}
    \bigg|\int_{\rone^d}d\textbf{v}_*\Phi(\textbf{v}_*)_{ij}&(f-f^{nn})(\textbf{v}-\textbf{v}_*)\bigg|\\
    \leq&\int_{\rone^d}d\textbf{v}_*|\textbf{v}_*|^\gamma\left||\textbf{v}_*|^2\delta_{ij}-(\textbf{v}_*)_i(\textbf{v}_*)_j\right||(f-f^{nn})(\textbf{v}-\textbf{v}_*)|\\
    \leq& \int_{\rone^d}d\textbf{v}_*|\textbf{v}_*|^{\gamma+2}|(f-f^{nn})(\textbf{v}-\textbf{v}_*)|= \int_{\rone^d\setminus B_{2R}(0)}d\textbf{v}_*+\int_{B_{2R}(0)}d\textbf{v}_*.
\end{align*}
Since the truncation velocity domain $V=B_R(0)$ is chosen sufficiently large, we bound the first integration as
\begin{align}\label{out_bound}
    \int_{\rone^d\setminus B_{2R}(0)}d\textbf{v}_* |\textbf{v}_*|^{\gamma+2}|&(f-\underbrace{f^{nn}}_{=0})(\textbf{v}-\textbf{v}_*)|\\
    \leq& G_1\int_{2R}^\infty dr\int_{u\in\partial S^{d-1}}dS\; r^{\gamma+2+d-1}e^{-\alpha|r\textbf{u}-\textbf{v}|^2}\notag\\
    \leq& dG_1\omega_d \int_{2R}^\infty dr\;\frac{r^{\gamma+2+d-1}}{e^{\alpha(r-|\textbf{v}|)^2}} <dG_1\omega_d\notag,
\end{align}
by \eqref{assum_decay} where $\omega_d$ is the volume of the unit ball in $\rone^d$. For the second integration, we have
\begin{align}\label{in_bound}
    \int_{B_{2R}(0)}d\textbf{v}_*|\textbf{v}_*|^{\gamma+2}&|(f-f^{nn})(\textbf{v}-\textbf{v}_*)|\\
    \leq& \left(\int_{B_{2R}(0)}d\textbf{v}_*|\textbf{v}_*|^{2\gamma+4}\right)^{\frac{1}{2}}\left(\int_{B_{2R}(0)}d\textbf{v}_*|(f-f^{nn})(\textbf{v}-\textbf{v}_*)|^2\right)^{\frac{1}{2}}\notag\\
    \leq& \left(\int_0^{2R} dr\;r^{2\gamma+4+d-1}\right)^{\frac{1}{2}}\notag\\
    &\bigg(\int_V d\textbf{v}_*|(f-f^{nn})(\textbf{v}_*)|^2+\int_{B_{2R}(v)\setminus V} d\textbf{v}_*|(f-f^{nn})(\textbf{v}_*)|^2\bigg)^{\frac{1}{2}}\notag\\
    \leq& M_3 \left(\|f-f^{nn}\|_{L^2_\textbf{v}(V)}^2+\epsilon^2\text{vol}(B_{2R}(v)\setminus V)\right)^{\frac{1}{2}}\notag,
\end{align}
when the index $\gamma$ satisfies $\gamma>-\frac{d}{2}-2$ ($\Leftrightarrow 2\gamma+4+d-1>-1$) for some positive constant $M_3$ by H\"older's inequality and using the assumption \eqref{assum_bdry}. Therefore, $A_1(t)$ is bounded as
\begin{align*}
  A_1(t)\leq& M_4\epsilon\left(\left(\|f-f^{nn}\|_{L^2_\textbf{v}(V)}^2+\epsilon^2\right)^{\frac{1}{2}}+\|D-D^{nn}\|_{op} + dG_1\omega_d\right) + M_2\epsilon\\
  \leq& \frac{{M_4}^2}{2}\epsilon^2+\frac{1}{2}\|f-f^{nn}\|_{L^2_\textbf{v}(V)}^2+\frac{\epsilon^2}{2}+M_4\epsilon(\varepsilon_D+dG_1\omega_d)+M_2\epsilon
\end{align*}
for some positive constant $M_4$ by Young's inequality. Also, by \eqref{D_D}, the second term $-A_2(t)$ also can be bounded as 
\begin{multline}
-A_2(t)\leq\left|\int_Vd\textbf{v}\nabla_\textbf{v}(f-f^{nn})\cdot(D(f)-D^{nn}(f^{nn}))\nabla f^{nn}\right|\\
-\int_Vd\textbf{v}\nabla_\textbf{v}(f-f^{nn})\cdot D(f)\nabla_\textbf{v}(f-f^{nn})\\
\leq M_5\sum_{1\leq i,j\leq d}\int_Vd\textbf{v}\int_{\rone^d}d\textbf{v}_*\left|\partial_{v_i}(f-f^{nn})(\textbf{v})(f-f^{nn})(\textbf{v}-\textbf{v}_*)\Phi(\textbf{v}_*)_{ij}\right|\\
+d^2M_5\|D-D^{nn}\|_{op}\|\nabla_\textbf{v}(f-f^{nn})\|_{L^1_\textbf{v}(V)}\\
-\sum_{1\leq i,j\leq d}\int_Vd\textbf{v}\partial_i(f-f^{nn})[D(f)]_{ij}\partial_j(f-f^{nn})
\end{multline}
for some positive constant $M_5$ since $\nabla f^{nn}$ is bounded on the compact domain $V$. Using the property \eqref{assum_ellipticity}, we have
\begin{equation}
    -\sum_{1\leq i,j\leq d}\int_Vd\textbf{v}\partial_i(f-f^{nn})[D(f)]_{ij}\partial_j(f-f^{nn})\leq -M_6K\|\nabla_\textbf{v}(f-f^{nn})\|_{L^2_\textbf{v}(V)}^2,
\end{equation}
for some positive constant $M_6$. Using Young's convolution inequality, we have
\begin{align}\label{young_convolution}
    \int_Vd\textbf{v}\int_{\rone^d}d\textbf{v}_*&\left|\partial_{v_i}(f-f^{nn})(\textbf{v})(f-f^{nn})(\textbf{v}-\textbf{v}_*)\Phi(\textbf{v}_*)_{ij}\right|\\
    \leq& \int_Vd\textbf{v}\int_{\rone^d}d\textbf{v}_*\left|\partial_{v_i}(f-f^{nn})(\textbf{v})(f-f^{nn})(\textbf{v}-\textbf{v}_*)|\textbf{v}_*|^{\gamma+2}\right|\notag\\
    =&\int_Vd\textbf{v}\int_{\rone^d\setminus B_{2R}(0)}d\textbf{v}_*\left|\partial_{v_i}(f-f^{nn})(\textbf{v})(f-f^{nn})(\textbf{v}-\textbf{v}_*)|\textbf{v}_*|^{\gamma+2}\right|\notag\\
    &+\int_Vd\textbf{v}\int_{B_{2R}(0)}d\textbf{v}_*\left|\partial_{v_i}(f-f^{nn})(\textbf{v})(f-f^{nn})(\textbf{v}-\textbf{v}_*)|\textbf{v}_*|^{\gamma+2}\right|\notag\\
    \leq& \epsilon\|\partial_{v_i}(f-f^{nn})\|_{L^1_\textbf{v}(V)}\notag\\
    &+ \|\partial_{v_i}(f-f^{nn})\|_{L^2_\textbf{v}(V)}\left(\|f-f^{nn}\|_{L^2_\textbf{v}(V)}+M_7\epsilon^2\right)\int_{B_{2R}(0)}d\textbf{v}|\textbf{v}|^{\gamma+2}\notag
\end{align}
for some positive constant $M_7$ in a similar way in \eqref{out_bound} and \eqref{in_bound}. When the index $\gamma$ satisfies $\gamma>-d-2$ ($\Leftrightarrow \gamma+2+d-1>-1$), the last integral term
\begin{equation}
     \int_{B_{2R}(0)}d\textbf{v}|\textbf{v}|^{\gamma+2}=\int_0^{2R} dr\;r^{\gamma+2+d-1}
 \end{equation}
is bounded. Therefore, we have
\begin{align*}
    -A_2(t)\leq& M_8\|\nabla_{\textbf{v}}(f-f^{nn})\|_{L^2_\textbf{v}(V)}\left(\|f-f^{nn}\|_{L^2_\textbf{v}(V)}+\epsilon+\|D-D^{nn}\|_{op}\right)\\
    &-M_6K\|\nabla_\textbf{v}(f-f^{nn})\|_{L^2_\textbf{v}(V)}^2\\
    \leq& 2\varepsilon_0\|\nabla_{\textbf{v}}(f-f^{nn})\|_{L^2_\textbf{v}(V)}^2+\frac{M_8^2}{4\varepsilon_0}\|f-f^{nn}\|_{L^2_\textbf{v}(V)}^2+\frac{M_8^2}{4\varepsilon_0}(\epsilon^2+\varepsilon_D^2)\\
    &-M_6K\|\nabla_\textbf{v}(f-f^{nn})\|_{L^2_\textbf{v}(V)}^2
\end{align*}
for some positive constant $M_8$ by H\"older's inequality and Young's inequality. Therefore, we can reduce \eqref{At_all} to
\begin{multline}\label{At_final}
     A(t) \leq \left(\frac{1}{2}+\frac{M_8^2}{4\varepsilon_0}\right)\|f-f^{nn}\|_{L^2_\textbf{v}(V)}^2+2\varepsilon_0\|\nabla_{\textbf{v}}(f-f^{nn})\|_{L^2_\textbf{v}(V)}^2\\
     -M_6K\|\nabla_\textbf{v}(f-f^{nn})\|_{L^2_\textbf{v}(V)}^2+C_1(\epsilon_0)\varepsilon_D+C_2(\epsilon_0)\epsilon.
\end{multline}
for some positive constants $C_1(\epsilon_0)$ and $C_2(\epsilon_0)$ which depend on $\epsilon_0$. On the other hands, we consider the $B(t)$ as
\begin{align*}
-B(t)\\
=&-\int_Vd\textbf{v}(f-f^{nn})\nabla_\textbf{v}\cdot\left[(F(f)-F^{nn}(f^{nn}))f^{nn}+F(f)(f-f^{nn})\right]\\
=&-\int_Vd\textbf{v}(f-f^{nn})\bigg[\nabla_\textbf{v}\cdot(F(f)-F^{nn}(f^{nn}))f^{nn}\\
&+(F(f)-F^{nn}(f^{nn}))\cdot\nabla_\textbf{v}f^{nn}+\nabla_\textbf{v}\cdot F(f)(f-f^{nn})+F(f)\cdot\nabla_\textbf{v}(f-f^{nn})\bigg]
\end{align*}
Since $f^{nn}$, $\nabla_\textbf{v}f^{nn}$, $F(f)$ and $\nabla_\textbf{v}\cdot F(f)$ is continuous function with respect to $v$ on the compact set $V$, we can bound $-B(t)$ as
\begin{multline}\label{Bt_all}
-B(t)\leq\\
M_9\bigg(\sum_{1\leq i\leq d}\int_Vd\textbf{v}|f-f^{nn}|\left|\partial_{v_i}(F(f)-F^{nn}(f^{nn}))_i+(F(f)-F^{nn}(f^{nn}))_i\right|\\
+ \|f-f^{nn}\|_{L^2_\textbf{v}(V)}^2 + \|f-f^{nn}\|_{L^2_\textbf{v}(V)}\|\nabla_{\textbf{v}}(f-f^{nn})\|_{L^2_\textbf{v}(V)}\bigg)
\end{multline}
for some positive constant $M_9$ by H\"older's inequality. For the first integral term on the right side of \eqref{Bt_all}, we have
\begin{multline}
    \int_Vd\textbf{v}|f-f^{nn}|\left|\partial_{v_i}(F(f)-F^{nn}(f^{nn}))_i+(F(f)-F^{nn}(f^{nn}))_i\right|\\
    \leq  \sum_{1\leq j\leq d}\int_Vd\textbf{v}\int_{\rone^d}d\textbf{v}_*\bigg|(f-f^{nn})(\textbf{v})(\partial_{v_i}\Phi_{ij}+\Phi_{ij})(\textbf{v}-\textbf{v}_*)\nabla_{v_j}(f-f^{nn})(\textbf{v}_*)\bigg|\\
    +\int_Vd\textbf{v}|f-f^{nn}|\left(\|F-F^{nn}\|_{op}+\|\nabla_{\textbf{v}}\cdot(F-F^{nn})\|_{op}\right)
\end{multline}
Therefore, by Young’s convolution inequality, we have
\begin{multline}
    -B(t)\\
    \leq M_{10}\|\nabla_{\textbf{v}}(f-f^{nn})\|_{L^2_\textbf{v}(V)}\left(\|f-f^{nn}\|_{L^2_\textbf{v}(V)}+\epsilon\right)+2M_9\|f-f^{nn}\|_{L^2_\textbf{v}(V)}^2+M_9\varepsilon_F^2
\end{multline}
for some positive constant $M_{10}$ in a similar way in \eqref{young_convolution} except that it holds when the index $\gamma$ satisfies $\gamma>-d-1$ ($\Leftrightarrow \gamma+1+d-1>-1$) due to the derivative term $\partial_{v_i}\Phi_{ij}$. Therefore, we have
\begin{multline}\label{Bt_final}
     -B(t)\\
     \leq \left(\frac{M_{10}^2}{4\varepsilon_0}+2M_9\right)\|f-f^{nn}\|_{L^2_\textbf{v}(V)}^2+2\varepsilon_0\|\nabla_{\textbf{v}}(f-f^{nn})\|_{L^2_\textbf{v}(V)}^2+C_3(\varepsilon_0)\varepsilon_F+C_4(\varepsilon_0)\epsilon.
\end{multline}
for some positive constants $C_3(\epsilon_0)$ and $C_4(\epsilon_0)$ which depend on $\epsilon_0$. Using the inequalities \eqref{At_final} and \eqref{Bt_final}, we can reduce \eqref{energy_eq} to
\begin{multline}
\frac{1}{2}\frac{d}{dt}\|f-f^{nn}\|_{L^2_\textbf{v}(V)}^2\leq \left(\frac{M_8^2+M_{10}^2}{4\varepsilon_0}+2M_9+1\right)\|f-f^{nn}\|_{L^2_\textbf{v}(V)}^2\\
+4\varepsilon_0\|\nabla_{\textbf{v}}(f-f^{nn})\|_{L^2_\textbf{v}(V)}^2-M_6K\|\nabla_\textbf{v}(f-f^{nn})\|_{L^2_\textbf{v}(V)}^2 +\frac{1}{2}\|d_{ge}(t,\textbf{v})\|_{L^2_\textbf{v}(V)}^2 \\
+ C_1(\epsilon_0)\varepsilon_D + C_3(\varepsilon_0)\varepsilon_F + (C_2(\epsilon_0)+C_4(\epsilon_0))\epsilon.
\end{multline}
By choosing a sufficiently small $\varepsilon_0$ satisfies $\varepsilon_0<\frac{M_6K}{4}$, we have
\begin{multline}\label{energy_ineq}
\frac{d}{dt}\overbrace{\|f-f^{nn}\|_{L^2_\textbf{v}(V)}}^{Y(t)\eqdef}\leq C_5\overbrace{\|f-f^{nn}\|_{L^2_\textbf{v}(V)}}^{Y(t)}+\|d_{ge}(t,\textbf{v})\|_{L^2_\textbf{v}(V)}^2\\
+C_6\varepsilon_D+C_7\varepsilon_F+C_8\epsilon.
\end{multline}
for some positive constants $C_5$, $C_6$, $C_7$ and $C_8$ which depend on $R$ and $d$ $(d=2,3)$. We rewrite the equation \eqref{energy_ineq} as
\begin{equation}
    Y'(t)-C_5Y(t)\leq  \|d_{ge}(t,\textbf{v})\|_{L^2_\textbf{v}(V)}^2+C_6\varepsilon_D+C_7\varepsilon_F+C_8\epsilon.
\end{equation}
By Gronwall's inequality, we have
\begin{multline}\label{gronwall}
    Y(t) \leq e^{C_5t}\bigg(Y(0)+\int_0^tds\;e^{-C_5s}\|d_{ge}(s,\textbf{v})\|_{L^2_\textbf{v}(V)}^2\\
    +\int_0^tds\;e^{-C_5s}\left(C_6\varepsilon_D+C_7\varepsilon_F+C_8\epsilon\right)\bigg).
\end{multline}
Note that $Y(0)=Loss_{IC}(f^{nn})$ and
\begin{equation}
    \int_0^tds\;e^{-C_5s}\|d_{ge}(s,\textbf{v})\|_{L^2_\textbf{v}(V)}^2\leq\int_0^Tds\;\|d_{ge}(s,\textbf{v})\|_{L^2_\textbf{v}(V)}^2=Loss_{GE}(f^{nn}, D^{nn}, F^{nn}).
\end{equation}
Therefore, equation \eqref{gronwall} implies that
\begin{multline}
    \|(f-f^{nn})(t,\cdot)\|^2_{L^\infty_t([0,T];L^{2}_\textbf{v}(V))}=\|Y(t)\|_{L^\infty_t([0,T])}\\
    \leq C(Loss_{IC}(f^{nn})+Loss_{GE}(f^{nn}, D^{nn}, F^{nn})+\varepsilon_D+\varepsilon_F+\epsilon),
\end{multline}
for some positive constant $C$ which depends only on $T$, $R$ and $d$ $(d=2,3)$. This completes the proof of Theorem \ref{thm2}.
\end{proof}



\bibliographystyle{amsplaindoi} 
\bibliography{bibliography}





\end{document}